\documentclass[12pt]{article}

\usepackage{geometry}
\usepackage{amsmath,amssymb,amsthm,comment,enumitem, array}
\usepackage{color}
\usepackage{parskip}
\usepackage{parskip}
\usepackage{setspace}
\usepackage{graphicx}
\usepackage{mathtools}
\usepackage{subfig}
\usepackage{comment}
\usepackage{mathabx}
\usepackage{mathrsfs}
\usepackage{setspace}
\usepackage{bm}
\usepackage[normalem]{ulem}
\usepackage{cancel}
\usepackage{csquotes}
\usepackage[colorlinks=true, allcolors=blue]{hyperref}
\hypersetup{colorlinks, citecolor=red, filecolor=black, linkcolor=blue, urlcolor=blue}
\usepackage{pgfplots}
\pgfplotsset{compat=1.18}
\usepackage{graphics}
\usepackage{graphicx}
\usepackage{tikz}
\usepackage{circuitikz}
\usepackage{tcolorbox}
\usepackage{enumitem}
\usepackage{subfiles}
\usepackage{ytableau}
\usepackage{thmtools}
\usepackage{thm-restate}
\usepackage{multicol}

\newcommand\bb{\mathbb}

\newcommand\CC{\bb{C}}

\newcommand\lb{\langle} 
\newcommand\rb{\rangle}

\newcommand\per{\text{per}}

\newcommand\sgn{\mathrm{sgn}}

\newtheorem{thm}{Theorem}
\newtheorem{defn}{Definition}
\newtheorem{lemma}{Lemma}
\newtheorem{prop}{Proposition}
\newtheorem{claim}{Claim}

\newtheorem{cor}{Corollary}

\newtheorem{conj}{Conjecture}

\def\john#1{\noindent
\textcolor{green}
{\textsc{(John:}
\textsf{#1})}}

\title{New constructions and bounds for nonabelian Sidon sets with applications to Tur\'an-type problems}
\author{John Byrne\thanks{Department of Mathematical Sciences, University of Delaware. \texttt{jpbyrne@udel.edu}} \and Michael Tait\thanks{Department of Mathematics \& Statistics, Villanova University. \texttt{michael.tait@villanova.edu}. Both authors were partially supported by NSF grant DMS-2245556}}
\date{\today}

\begin{document}
\maketitle

\begin{abstract}
    An $S_k$-\textit{set} in a group $\Gamma$ is a set $A\subseteq\Gamma$ such that $\alpha_1\cdots\alpha_k=\beta_1\cdots\beta_k$ with $\alpha_i,\beta_i\in A$ implies $(\alpha_1,\ldots,\alpha_k)=(\beta_1,\ldots,\beta_k)$. An $S_k'$-\textit{set} is a set such that $\alpha_1\beta_1^{-1}\cdots\alpha_k\beta_k^{-1}=1$ implies that there exists $i$ such that $\alpha_i=\beta_i\text{ or }\beta_i=\alpha_{i+1}$. We give explicit constructions of large $S_k$-sets in the group $S_n$ and $S_2$-sets in $S_n\times S_n$ and $A_n\times A_n$. We give probabilistic constructions for `nice' groups which obtain large $S_2$-sets in $A_n$ and $S_2'$-sets in $S_n$. We also give upper bounds on the size of $S_k$-sets in certain groups, improving the trivial bound by a constant multiplicative factor. We describe some connections between $S_k$-sets and extremal graph theory. In particular, we determine up to a constant factor the minimum outdegree of a digraph which guarantees even cycles with certain orientations. As applications, we improve the upper bound on Hamilton paths which pairwise create a two-part cycle of given length, and we show that a directed version of the Erd\H{o}s-Simonovits compactness conjecture is false.
\end{abstract}

\section{Introduction} \label{Introduction}

\subsection{Background}

A \textit{Sidon sequence} in $[n]$ is a subset $A\subseteq\mathbb N$ such that the pairwise sums $a+b$ with summands taken from $A$ are all different, i.e.
$$\forall a,b,c,d\in A\,  a+b=c+d\Longrightarrow\{a,b\}=\{c,d\}.$$
This notion was introduced by Sidon \cite{sidon1932satz} in his work on Fourier analysis. Erd\H{o}s and Tur\'an \cite{erdos1941problem} proved that the maximum size $\Phi(n)$ of a Sidon sequence in $[n]$ satisfies $(1/\sqrt 2-o(1))\sqrt n<\Phi(n)<(1+o(1))\sqrt n$ and it was later shown that $\Phi(n) \sim \sqrt{n}$ \cite{bosechowla}. Since then many variants and generalizations of this problem have been studied and there is great interest in bounding the maximum size of a Sidon set in a given group. For further reading we refer to \cite{bajnok2018additive} and \cite{o2004complete}.

In this paper we are concerned with Sidon sets and their generalizations in arbitrary, possibly nonabelian groups, which were introduced by Babai and S\'os \cite{babai1985sidon}:

\begin{defn}
    Let $\Gamma$ be a group. We say that $A\subseteq\Gamma$ is a \textit{Sidon set of the first kind} if
    $$\alpha\beta=\gamma\delta$$
    with $\alpha,\beta,\gamma,\delta\in A$ implies that $|\{\alpha,\beta,\gamma,\delta\}|\le 2$. We say that $A$ is a \textit{Sidon set of the second kind} if
    $$\alpha\beta^{-1}=\gamma\delta^{-1}$$
    with $\alpha,\beta,\gamma,\delta\in A$ implies $|\{\alpha,\beta,\gamma,\delta\}|\le 2$. 
\end{defn}

Observe that if $\Gamma$ is abelian then these two conditions are equivalent. The authors of \cite{babai1985sidon} used probabilistic methods to construct large Sidon sets of both kinds in general groups.

\begin{thm}[Babai-S\'os \cite{babai1985sidon}] \label{Babai Sos lower bound}
Let $\Gamma$ be a group and $W\subseteq\Gamma$ be finite. Then $W$ contains Sidon sets of both kinds, of size
$(c+o(1))|W|^{1/3}$, where $c=3\cdot2^{1/3}/8>0.47247$.
\end{thm}
Godsil and Imrich \cite{godsil1987embedding} improved the constant to $(2/(7+4\sqrt 3))^{1/3}> 0.52365$ for Sidon sets of the first kind and $1/(2+\sqrt 3)^{1/3}>0.64468$ for Sidon sets of the second kind.

If $\Gamma$ is abelian, we say $A\subseteq\Gamma$ is a $B_k[g]$-set ($B_k$-set if $g=1$) if for any $\mu\in\Gamma$, there is at most one multiset $\{\alpha_1,\ldots,\alpha_k\}$ with $\alpha_i\in A$ such that $\alpha_1+\cdots+\alpha_k=\mu$. Odlyzko and Smith \cite{odlyzko1995nonabelian} introduced the following non-abelian analogue of $B_k$-sets.

\begin{defn}
    Let $\Gamma$ be a group. We say $A\subseteq\Gamma$ is a (nonabelian) $S_k$-set if whenever $$\alpha_1\cdots\alpha_k=\beta_1\cdots\beta_k$$
    with $\alpha_i,\beta_i\in A$, we have
    $$(\alpha_1,\ldots,\alpha_k)=(\beta_1,\ldots,\beta_k).$$
\end{defn}

An $S_2$-set is a Sidon set of the first kind but the converse is not  necessarily true. One may generalize $S_k$-sets to a nonabelian analogue of $B_k[g]$-sets:

\begin{defn}
    Let $\Gamma$ be a group. We say $A\subseteq\Gamma$ is an $S_k[g]$-set if for any $\mu\in\Gamma$ there are at most $g$ words $(\alpha_1,\ldots,\alpha_k)$ such that $\alpha_1\cdots\alpha_k=\mu$.
\end{defn}

Note that $\Gamma$ being nonabelian allows us to impose the stronger condition of the equality of the words $(\alpha_1,\ldots,\alpha_k)$ and $(\beta_1,\ldots,\beta_k)$ rather than of the multisets $\{\alpha_1,\ldots,\alpha_k\}$ and $\{\beta_1,\ldots,\beta_k\}$. This is important for the applications of Sidon-type sets to extremal graph theory. Given a set $A\subseteq\Gamma$, its \textit{Cayley graph} $\mathrm{Cay}(\Gamma,A)$ is the digraph with vertex set $\Gamma$ where $\alpha\beta$ is an edge whenever $\alpha^{-1}\beta\in A$; its \textit{bipartite Cayley graph} $\mathrm{BCay}(\Gamma,A)$ is the undirected graph with vertex set $\Gamma\times\{0,1\}$ whose edges are $\{(\alpha,0),(\alpha\beta,1)\}$ for $\alpha\in\Gamma,\beta\in A$. It is well-known that the bipartite Cayley graph of a $B_2$-set is $C_4$-free: see \cite{tait2013sidon,daza2018sidon} for applications of this connection to extremal graph theory. Unfortunately, when $k\ge 3$ the bipartite Cayley graph of a $B_k$-set contains a $C_{2k}$. However, as described in \cite{odlyzko1995nonabelian} there is hope of constructing large $C_{2k}$-free graphs using another non-abelian analogue of $B_k$-sets.

\begin{defn}
    Let $\Gamma$ be a group. We say $A\subseteq\Gamma$ is an $S_k'$-set if whenever
    $$\alpha_1\beta_1^{-1}\cdots\alpha_k\beta_k^{-1}=1$$
    with $\alpha_i,\beta_i\in A$, we have for some $i$ that $\alpha_i=\beta_i$ or $\beta_i=\alpha_{i+1}$. 
\end{defn}

An $S_2'$-set is a Sidon set of the second kind but the converse is not true. However, observe that the bipartite Cayley graph of an $S_k'$-set is $C_{2k}$-free. A partial converse holds: if $G$ is a (bipartite) Cayley graph with girth greater than $2k$, then the generating set is an $S_k'$-set. This means that constructions of high-girth Cayley graphs can be phrased in terms of $S_k'$-sets; for example, the Ramanujan graphs of Lubotzky, Phillips, and Sarnak \cite{lubotzky1988} provide a construction of $S_k'$-sets in $\mathrm{PSL}(2,q)$ and $\mathrm{PGL}(2,q)$.

Let $M_{k,g}(\Gamma)$ denote the maximum size of an $S_k[g]$-set in $\Gamma$, and let $M_k'(\Gamma)$ denote the maximum size of a $S_k'$-set in $\Gamma$. When $g=1$, we just write $M_k(\Gamma)$. If $A$ is an $S_k$-set then the words in $A^k$ give distinct products, so we have the trivial upper bound $M_k(\Gamma)\le|\Gamma|^{1/k}$. More generally, $M_{k,g}(\Gamma)\le (g|\Gamma|)^{1/k}$. For $S_k'$-sets the general upper bound is not so immediate. Let $A\subseteq\Gamma$ be an $S_k'$-set. Then $\mathrm{BCay}(\Gamma,A)$ is a $C_{2k}$-free graph on $2|\Gamma|$ vertices with $|\Gamma||A|$ edges. The even cycle theorem \cite{bondy1974cycles} gives $|\Gamma||A|=O(|\Gamma|^{1+1/k})$, so $M_k'(\Gamma)=O(|\Gamma|^{1/k})$. The authors of \cite{odlyzko1995nonabelian} constructed $S_k$-sets in certain infinite families of groups whose size is within a constant factor of the upper bound:

\begin{thm}[Odlyzko-Smith \cite{odlyzko1995nonabelian}] \label{Odlyzko result}
For each integer $k$ at least $2$, and any prime $p$ with $k|(p-1)$, a nonabelian group $G$ of order $|G|=(p^k-1)k$ exists which contains a nonabelian $S_k$-set $S$ of cardinality $(p-1)/k$.
\end{thm}

Our aims in this paper are twofold. First, we give lower and upper bounds on $M_k(\Gamma)$ and $M_k'(\Gamma)$ in various groups. We list these results in \autoref{Subsection Sidon sets}. Second, we establish connections between $S_k$-sets and some problems in extremal graph theory, and we study these problems in their own right. We list these results in \autoref{Subsection extremal graph theory}.

\subsection{Results on Sidon sets} \label{Subsection Sidon sets}

Our lower bounds on $M_k(\Gamma)$ will focus on the groups $S_n,S_n\times S_n$, and $A_n\times A_n$, where $S_n$ and $A_n$ are the symmetric and alternating groups on $n$ letters, respectively. There is a large literature on extremal problems for the symmetric group, including properties of its Cayley graphs. For example, Helfgott and Seress \cite{helfgott2014diameter} showed that if $\Gamma=S_n$ or $\Gamma=A_n$ then for any set $A\subseteq\Gamma$ which generates $\Gamma$, every element of $\Gamma$ can be expressed as a product of $\mathrm{exp}((\log\log|\Gamma|)^{O(1)})$ elements of $A\cup A^{-1}$. Keevash and Lifshitz \cite{KL} obtained results on combinatorial properties of the symmetric group, including diameter of the Cayley graph of a dense generating set and the size of subsets avoiding the equation $\alpha\beta=\gamma^2$. Recently Keevash, Lifshitz, and Minzer \cite{KLM} determined the maximum product-free subsets of $A_n$. Illingworth, Michel, and Scott \cite{IMS} studied similar problems in infinite groups. Our first result is a lower bound on $M_k(S_n)$.

\begin{thm} \label{Sk sets in Sn lower bound} For all $k$, we have
$$M_k(S_n)=(n!)^{1/k+O(1/\log n)}.$$
\end{thm}

The idea of \autoref{Sk sets in Sn lower bound} is to use the $S_k$-sets of \autoref{Odlyzko result} and consider the permutations of $\Gamma$ which map each $\alpha$ to some $\alpha\beta$, where $\beta$ belongs to the $S_k$-set. The Egorychev-Falikman theorem \cite{egorychev,falikman}, which provides a lower bound on the permanent of a doubly stochastic matrix, allows us to estimate the number of such permutations. 

Observe that if $A_1\subseteq\Gamma_1$ and $A_2\subseteq\Gamma_2$ are $S_k$-sets, then $A_1\times A_2$ is an $S_k$-set in $\Gamma_1\times\Gamma_2$. This is a notable contrast to $B_k$-sets. As a consequence, \autoref{Sk sets in Sn lower bound} gives that $M_k(S_n\times S_n)\ge(n!)^{2/k-O(1/\log n)}.$ In the case $k=2$, we provide a better construction whose size can be computed exactly and which is optimal up to a factor of $n$.

\begin{thm} \label{Sn times Sn result}
    For every $n$ we have
    \begin{itemize}
        \item[(a)] $M_2(S_n\times S_n)\ge(n-1)!$
        \item[(b)] $M_{2,n}(S_n\times S_n)\ge n!$
        \item[(c)] $M_2(A_n\times A_n)\ge(n-1)!/2$
        \item[(d)] $M_{2,n}(A_n\times A_n)\ge n!/2.$
    \end{itemize}
\end{thm}

Inspired by the construction of Sidon sets in elementary abelian groups of order $q^2$ \cite{lindstrom1969determination,babai1985sidon} (which are themselves based on the original construction of Erd\H{o}s and Tur\'an \cite{erdos1941problem}), our constructions are loosely of the form $\{(\alpha,f(\alpha)):\alpha\in\Gamma\}$ where $f:\Gamma\to\Gamma$. However, in nonabelian groups we cannot use polynomials so we require other tools to find a function $f$ which gives a Sidon set. In the case of $S_n$ we are able to exploit the relationship between cycle structure and conjugacy. \autoref{Sk sets in Sn lower bound} and \autoref{Sn times Sn result} give not only an explicit construction of $S_2$-sets in these groups but also, to our knowledge, the first improvement over \cite{godsil1987embedding} on Sidon sets of the first kind in these groups. In \autoref{Conjugacy section} we also generalize parts (b) and (d) of \autoref{Sn times Sn result} to any group with a large conjugacy class.

We also consider Sidon sets of the second kind in $S_n$. Unfortunately, neither the idea of \autoref{Sk sets in Sn lower bound} nor its graph-theoretic generalization work here. That is, taking permutations from a $C_4$-free graph does not give rise to a Sidon set of the second kind in any direct way (see \autoref{Concluding remarks} for details). We make do with a general probabilistic lower bound, extending \autoref{Babai Sos lower bound} to $S_2$-sets and $S_2'$-sets. We did not attempt to optimize the constants.

\begin{prop} \label{Probabilistic constructions}
We have the following lower bounds on $M_2(\Gamma)$ and $M_2'(\Gamma)$.
\begin{itemize}
    \item[(a)] Suppose that a group $\Gamma$ has a set $B$ of size $b$ where any distinct $\beta_1,\beta_2\in B$ satisfy $\beta_1^2\ne\beta_2^2$ and $\beta_1\beta_2\ne\beta_2\beta_1$. Then $M_2(\Gamma)\ge(0.39+o(1))b^{1/3}.$
    \item[(b)] Suppose $\Gamma$ has exactly $i$ involutions. If $i=o(|\Gamma|^{2/3})$, then $M_2'(|\Gamma|)\ge (0.39+o(1))|\Gamma|^{1/3}$. If $i=\Omega(|\Gamma|^{2/3})$, then $M_2'(\Gamma)=\Omega(|\Gamma|/i)$.
\end{itemize}
\end{prop}

We give two applications. First, we note that $S_n$ has $(n!)^{1/2+o(1)}$ involutions, so \autoref{Probabilistic constructions} (b) gives $M_2'(S_n)=\Omega(n!^{1/3})$. By taking translations it follows that also $M_2'(A_n)=\Omega(n!^{1/3})$. Second, we consider $M_2(A_n)$. Let $B$ be a set of $n$-cycles or $(n-1)$-cycles fixing the same element (so that their sign is even) where $\pi\in B\Longrightarrow\pi^k\not\in B$ for $k\ne 1$. We can always find at least $(n-2)!/n$ such cycles. Since the sign of the cycles is even, we have $\beta_1^2\ne\beta_2^2$ for $\beta_1,\beta_2\in B$. It is well-known that two cycles $\pi,\sigma$ commute if and only if they are disjoint or $\sigma\in\lb\pi\rb$. Thus, $\beta_1\beta_2\ne\beta_2\beta_1$ for $\beta_1,\beta_2\in B$. Therefore, $M_2(A_n)\ge(n!)^{1/3-o(1)}$. To our knowledge these lower bounds are the best known, although we suspect the correct exponent is $1/2-o(1)$ in both cases.

We note that, in general, it is harder to give probabilistic lower bounds for $S_k$-sets or $S_k'$-sets than for Sidon sets. For example, the largest number $b$ attainable for \autoref{Probabilistic constructions} (a) can vary between 1 and $|\Gamma|^{1-o(1)}$ depending on the structure of the group. 

Finally we present upper bounds on the size of $S_k$-sets and $S_k'$-sets. Dimovski \cite{dimovski1992groups} proved that equality can never hold in the trivial bound on $S_k$-sets, i.e. $M_k(\Gamma)<|\Gamma|^{1/k}$
whenever $|\Gamma|>1$. Our main upper-bound result generalizes the argument of \cite{dimovski1992groups} to show that a kind of stability sometimes holds.

\begin{thm} \label{Sk upper bound}
    For any $h$ and any even $k$, there is $\varepsilon>0$ such that any sufficiently large group $\Gamma$ containing a normal abelian subgroup $H$ with $|\Gamma:H|=h$ satisfies
    $$M_k(\Gamma)\le(1-\varepsilon)|\Gamma|^{1/k}.$$
\end{thm}

In \autoref{Upper bound section} we prove various other upper bounds on $M_k(\Gamma)$ and $M_k'(\Gamma)$ when some information about the structure of $\Gamma$ is known.

\subsection{Results on extremal graph theory} \label{Subsection extremal graph theory}

Our first result in this category demonstrates another connection between Sidon sets and extremal graph theory, in the `reverse' direction: given a $C_{2k}$-free graph on $n$ vertices, one can construct an $S_k$-set in $S_n$.

\begin{thm} \label{hamilton cycles result}
    Suppose $G$ is a graph on $n$ vertices with girth at least $2k+1$ that contains $h$ Hamilton cycles. Then $M_k(S_n)\ge h/2^{n-1}$.
\end{thm}

Note that \autoref{hamilton cycles result} never improves \autoref{Sk sets in Sn lower bound} and only provides an equally good bound in the cases $k=2,3,5$ (in these cases, one can use pseudorandom constructions of extremal high-girth graphs to count the Hamilton cycles, see \cite{byrne2024improved}). However we find the result to be interesting for two reasons. First, it demonstrates that the connection between additive combinatorics and $C_{2k}$-free graphs sometimes goes in both directions. Second, it potentially implies the existence of many more distinct maximal $S_k$-sets than is guaranteed by \autoref{Sk sets in Sn lower bound}, owing to the increased flexibility of graphs as compared with Sidon sets. 

Next we consider the relationship between $S_k$-sets and directed graphs. Some terminology is required: let $\mathcal F_k$ be the set of all digraphs which are the union of two distinct directed walks of length $k$ with the same initial and same terminal vertices, let $C_{k,k}$ be the graph consisting of two vertices $x,y$ joined by two internally disjoint paths on $k$ edges, each oriented from $x$ to $y$, and let $\mathcal C_{k,k}=\{C_{2,2},\ldots,C_{k,k}\}$. If $\mathcal F$ is a family of (directed) graphs then $\mathrm{ex}(n,\mathcal F)$ is the maximum number of edges in a (directed) graph with no subgraph isomorphic to $\mathcal F$.

Huang and Lyu \cite{huang2020extremal} showed that $\mathrm{ex}(n,C_{2,2})=n^2/4+n+O(1)$ and determined the extremal digraphs for $n\ge 13$. Later \cite{huang2024extremal}, they determined $\mathrm{ex}(n,F)$ for large $n$ where $F$ is a particular orientation of $\Theta_{\ell,\ldots,\ell}$, in particular $\mathrm{ex}(n,C_{\ell,\ell})=n^2/4+O(n)$. Wu \cite{wu20100} showed that $\mathrm{ex}(n,\mathcal F_2)=n^2/4+n+O(1)$ and determined the extremal digraphs. Huang, Lyu, and Qiao \cite{huang2019turan} showed that for $k\ge 4$, $\mathrm{ex}(n,\mathcal F_k)=n^2/2-\lfloor n/k\rfloor^2/2+O(n)$ and determined the extremal digraphs when $k\ge 5$ and $n\ge k+5$. Huang and Lyu \cite{huang2022extremal} showed that $\mathrm{ex}(n,\mathcal F_3)=\lfloor n^2/3\rfloor+1$ and determined the extremal digraphs for $n\ge 16$. 

In all these results, the extremal graphs have a very unbalanced outdegree sequence, for example in \cite{huang2024extremal} they are obtained by some small modification of $K_{n/2,n/2}$ with edges oriented consistently from one part to the other. Thus, it is natural to ask how the problem changes when considering a minimum-degree rather than size condition. Let $m^+(n,\mathcal F)$/$m^-(n,\mathcal F)$/$m^0(n,\mathcal F)$ be the largest possible minimum outdegree/indegree/semidegree of an $n$-vertex $\mathcal F$-free digraph\footnote{In \cite{kelly2010cycles} the notation $\delta_{di}(\ell,n)$ was introducted for function we call $m^0(n,C_\ell)$, where $C_\ell$ is the strongly connected orientation of the $\ell$-cycle.}. As we show below, when considering even cycles these extremal functions resemble the undirected Tur\'an number $\mathrm{ex}(n,C_{2k})$ more closely than the directed Tur\'an number $\mathrm{ex}(n,C_{k,k})$. Kelly, Kuhn, and Osthus \cite{kelly2010cycles} showed that for any cycle $C$ such that $t(C)=0$ (meaning the number of forward edges in $C$ equals the number of backward edges; see \autoref{Notation}) one has $m^0(n,C)=o(n)$. We determine the order of magnitude of $m^0(n,\mathcal F)$ for certain families of forbidden cycles. (Note that if $C$ is the antidirected $C_{2\ell}$ with no directed path on three vertices, it is not too difficult to show that $m^+(n,C),m^-(n,C),m^0(n,C)=\Theta(\mathrm{ex}(n,C_{2\ell})/n)$; see also Conjecture 6.2 in \cite{zhou2023turan}.)

\begin{thm} \label{Cll girth upper bound}
    We have
    $$\left(\frac{1}{k^{1+1/k}}-o(1)\right)n^{1/k}\le m^0(n,\mathcal F_k)\le m^+(n,\mathcal C_{k,k})\le(2k+o(1))m^+(n,\mathcal F_k)\le(2k+o(1))n^{1/k}.$$
\end{thm}

The connection to $S_k$-sets appears in the first inequality above: the construction is the Cayley graph of an $S_k$-set in \autoref{Odlyzko result}.

For undirected graphs, the upper bounds $\mathrm{ex}(n,\{C_3,\ldots,C_{2k}\}),\mathrm{ex}(n,C_{2k})=O(n^{1+1/k})$ \cite{bondy1974cycles} are the best known and for $k=2,3,5$ there are matching lower bounds for both functions \cite{furedi1996number,benson1966minimal}. Somewhat surprisingly, in the directed case we find that forbidding only a single $C_{\ell,\ell}$ changes the problem significantly.

\begin{thm} \label{Cll lower bound}
    For any $\ell\ge 2$ we have
    $$\left(\frac{1}{(2\ell-2)^{1/2}}-o(1)\right)n^{1/2}\le m^0(n,C_{\ell,\ell})\le m^+(n,C_{\ell,\ell})\le(2\ell+o(1))n^{1/2}.$$
\end{thm}

The construction for the case $\ell=2$ of \autoref{Cll lower bound} can be used to construct large $S_2$-sets in $S_n$ and in fact improves case $k=2$ of \autoref{Sk sets in Sn lower bound} by an exponential factor. Since this improvement would be hidden in the error term $O(1/\log n)$, we skip the details. Another interesting application concerns $C_{2\ell}$-creating Hamilton paths. Let $\hat{M}(n,\ell)$ be the maximum number of Hamilton paths on $[n]$ with the property that given any two of them, there is a subpath of one and a subpath of the other such that the union of these subpaths is a copy of $C_\ell$. 
Cohen, Fachini and K\"orner \cite{cohen2017path} proved that $\hat{M}(n,4)\geq (n!)^{1/2+O(1/\log n)}$ and Harcos and Solt\'esz \cite{HS} proved that $\hat{M}(n,4)\leq (n!)^{1/2+O(1/\log n)}$ . For general even $\ell$, the best lower and upper bounds we are aware of are $$(n!)^{1/\ell-O(1/\log n)}\le\hat M(n,\ell)\le(n!)^{1-\frac{2}{3\ell}+O(1/\log n)}$$ 
which follow from \cite{soltesz2020even} and \cite{byrne2024improved} respectively. Using the construction in \autoref{Cll lower bound}, we are able to improve the upper bound.

\begin{cor} \label{Path separation application}
For even $\ell\ge 4$, we have
$$\hat M(n,\ell)\le(n!)^{1/2+O(1/\log n)}.$$
\end{cor}

Our final application concerns the following conjecture of Erd\H{o}s and Simonovits. Counterexamples are known to the original form of the conjecture in \cite{erdHos1982compactness}, so we state the modified version discussed in \cite{wigdersonerdos}.

\begin{conj}[Erd\H{o}s-Simonovits \cite{erdHos1982compactness}] \label{Compactness conjecture}
    For every finite collection $\mathcal F$ of graphs which contains no forest, there exists some $H\in\mathcal F$ and some $c>0$ so that
    $$\mathrm{ex}(n,\mathcal F)\ge c\cdot\mathrm{ex}(n,H)$$
    for all $n$.
\end{conj}

Comparing \autoref{Cll girth upper bound} and \autoref{Cll lower bound}, the finite family of graphs $\mathcal C_{k,k}$ satisfies $m^0(n,H)/m^0(n,\mathcal C_{k,k})\to\infty$ for every $H\in\mathcal C_{k,k}$. Thus, the version of \autoref{Compactness conjecture} obtained by replacing graphs with digraphs and $\mathrm{ex}$ with $m^0$ is false.

\section{Notation and definitions} \label{Notation}

Our directed graphs (digraphs) may have opposite edges but no parallel edges or loops. If $v\in V(G)$ we write $N^+(v)=\{u\in V(G):(v,u)\in E(G)\}$ and $N^-(v)=\{u\in V(G):(u,v)\in E(G)\}$; we write $d^+(v)$ for its outdegree $|N^+(v)|$ and $d^-(v)$ for its indegree $|N^-(v)|$, and we write $\delta^+(G)=\min\{d^+(v):v\in V(G)\}$, $\Delta^+(G)=\max\{d^+(v):v\in V(G)\}$ and similarly for the indegree. The \textit{minimum semidegree} of $G$ is $\delta^0(G)=\min\{\delta^+(G),\delta^-(G)\}$. A \textit{directed walk of length} $k$ in $G$ is a sequence of vertices $v_0\cdots v_k$ such that $(v_i,v_{i+1})\in E(G)$ for every $0\le i\le k-1$. A \textit{cycle of length} $k$ in $G$ is any cycle of length $k$ in the underlying graph of $G$. Given a closed walk $W=v_0e_0v_1e_1\cdots v_{k-1}e_{k-1}v_0$ in the underlying graph of a directed graph, its \textit{type} $t(W)$ is the absolute value of
$$|\{i:e_i=(v_i,v_{i+1})\}|-|\{i:e_i=(v_{i+1},v_i)\}|$$
with the sum $i+1$ taken modulo $k$, in other words it is the `net number of forward steps' in the walk. Given subsets $U_1,\ldots,U_k\subseteq V(G)$, we write $G[U_1,\ldots,U_k]$ for the graph with vertex set $U_1\cup\cdots\cup U_k$ containing all edges of $G$ directed from some $U_i$ to $U_{i+1}$, $1\le i\le k-1$. We define $E(U,W):=E(G[U,W])$ and $e(U,W)=|E(U,W)|$.

Given a set $X$, let $S_X$ denote the symmetric group on $X$. For a group $\Gamma$, $\gamma\in\Gamma$ and $A\subseteq\Gamma$, we define $\gamma A=\{\gamma\alpha:\alpha\in A\}$.

\section{Constructions using permanents} \label{main section}

\subsection{Proof of \autoref{hamilton cycles result}}
    Orient each edge of $G$ uniformly and independently, to obtain a random directed graph $G'.$ Say that $G'$ \textit{respects} a Hamilton cycle $H=v_0\cdots v_n$ if for all $i=0,\ldots,n-1$
    $$(v_i,v_{i+1})\in E(G')$$
    where the addition is taken modulo $n$. Since there are $2^n$ possible orientations of the edges of $H$ and 2 of them respect $H$, we have
    $$\mathbb P[G'\text{ respects }H]=1/2^{n-1}.$$
    Therefore,
    $$\mathbb E[|\{H:G'\text{ respects }H\}|]=h/2^{n-1}.$$
    Taking some orientation which respects at least as many Hamilton cycles as the expectation, we obtain a family $\mathcal H$ of at least $h/2^{n-1}$ directed Hamilton cycles. To each of these we associate the cyclic permutation $\pi_{H}\in S_n$ such $\pi_{H}(i)=j$ if $(i,j)\in E(H).$ These permutations are all distinct, so if $A=\{\pi_{H}:H\in\mathcal H\}$ then $|A|\ge h/2^{n-1}$.

    Now suppose $\alpha_1,\ldots,\alpha_k,\beta_1,\ldots,\beta_k\in A$ satisfy
    $$\alpha_k\cdots\alpha_1=\beta_k\cdots\beta_1.$$
    Let $i\in [n]$. For $\ell\in[0,k]$, let $x_\ell=(\alpha_\ell\cdots\alpha_1)(i)$ and $y_\ell=(\beta_\ell\cdots\beta_1)(i)$, so that $x_0=y_0=i$ and $x_k=y_k=(\alpha_k\cdots\alpha_1)(i).$ Since $G'$ has no opposite edges and the $\alpha_\ell,\beta_\ell$ are cyclic permutations, there is no pausing or backtracking:
    $$x_\ell\not\in\{x_{\ell-1},x_{\ell-2}\},\ y_\ell\not\in\{ y_{\ell-1},y_{\ell-2}\},\ \ell=2,\ldots,k.$$
    This implies that, if for some $\ell<\ell'$ we have $x_\ell=x_{\ell'}$, then $G[\{x_\ell,\ldots,x_{\ell'}\}]$ contains a cycle, which contradicts that the girth of $G$ is at least $2k+1$. Thus, $x_0,\ldots,x_k$ are all distinct and similarly so are $y_0,\ldots,y_k$. Moreover, $y_1=x_1$, for otherwise $x_k=y_k$ implies that $G[\{x_0,\ldots,x_k,y_0,\ldots,y_k\}]$ contains a cycle, contradicting that the girth of $G$ is at least $2k+1$. Thus $\alpha_1(i)=\beta_1(i)$, and this holds for all $i$ so that $\alpha_1=\beta_1$. We obtain
    $$\alpha_k\cdots\alpha_2=\beta_k\cdots\beta_2$$
    and repeating the argument $k$ times proves that for all $\ell$, $\alpha_\ell=\beta_\ell$.
\qed

\subsection{Proof of \autoref{Sk sets in Sn lower bound}}

We only need to prove the lower bound. Suppose $|\Gamma|=n$ and $A\subseteq\Gamma$ is an $S_k$-set of size $a$. Let
$$A'=\{\pi\in S_\Gamma:\forall x\in\Gamma\ \pi(x)\in xA\}.$$
Let $M$ be the $\Gamma\times\Gamma$ matrix where $M_{xy}=1$ if $x^{-1}y\in A$ and $M_{xy}=0$ otherwise. Then $A'$ is the set of permutations $\pi$ satisfying $M_{x\pi(x)}=1$ for all $x\in\Gamma$, and so $|A'|=\mathrm{per}(M)$. The matrix $M/a$ is doubly stochastic, so we can estimate $\mathrm{per}(M/a)$ using the Egorychev-Falikman theorem:
\begin{thm}[Egorychev-Falikman \cite{egorychev,falikman}] \label{Egorychev Falikman}
If $M$ is an $n\times n$ doubly stochastic matrix, then 
$$\mathrm{per}(M)\ge\frac{n!}{n^n}$$
with equality if and and only if $M$ is the constant matrix $n^{-1}J$.
\end{thm}
We obtain 
$$\mathrm{per}(M)\ge a^n\per(M/a)\ge a^n\frac{n!}{n^n}\ge a^{n-O(n/\log n)},$$
where the last inequality holds as long as $a\geq n^\epsilon$ (which it will be as we will obtain $A$ using Theorem \ref{Odlyzko result}). Now we claim that $A'$ is an $S_k$-set in $S_n$. If $\alpha_1\cdots\alpha_k=\beta_1\cdots\beta_k$ with $\alpha_i,\beta_i\in A'$ then
$$\forall x\in\Gamma\ \alpha_1\cdots\alpha_k(x)=\beta_1\cdots\beta_k(x).$$
By the definition of $A'$, there exist $a_1,\ldots,a_k,b_1,\ldots,b_k\in A$ such that $\alpha_k(x)=xa_k$, $\beta_k(x)=xb_k$, etc. so that
$$\begin{aligned}xa_k\cdots a_1&=xb_k\cdots b_1\\
a_k\cdots a_1&=b_k\cdots b_1\\
(a_k,\ldots,a_1)&=(b_k,\ldots,b_1).
\end{aligned}$$ 
In particular, $a_k=b_k$ implies $\alpha_k(x)=\beta_k(x)$. This holds for all $x\in\Gamma$, so $\alpha_k=\beta_k$ and thus $\alpha_1\cdots\alpha_{k-1}=\beta_1\cdots\beta_{k-1}$. Repeating this argument $k$ times, using the fact that the $S_k$-set $A$ is also an $S_{\ell}$-set for $\ell<k$, we find that $(\alpha_1,\ldots\alpha_k)=(\beta_1,\ldots,\beta_k)$.

We have shown that whenever such $\Gamma,A$ exist for given $n$, we have
$$M_k(S_n)\ge a^{n-O(n/\log n)}.$$
To obtain good $\Gamma,A$ we apply \autoref{Odlyzko result}. If $n=(p^k-1)k$ for some prime $p$ with $k|(p-1)$, we may take $a\ge cn^{1/k}$ (where $c$ depends only on $k$). Thus for such $n$,
$$M_k(S_n)\ge (cn^{1/k})^{n+O(n/\log n)}=n^{n/k+O(n/\log n)}=(n!)^{1/k+O(1/\log n)}.$$
Now let $n\in\mathbb N$ be arbitrary. We refer to the following density-of-primes result which will also be useful later.
\begin{thm}[Baker-Harman-Pintz \cite{baker1996exceptional}] \label{Baker Harman Pintz}
    Let $\pi(x;q,a)$ denote the number primes $p\le x$ with $p\equiv a\pmod q$. If $(a,q)=1$, $x^{0.55+\varepsilon}\le M\le x/\log x$, $q\le\log^Ax$ (for constant $A>0$) and $x$ is large enough,
    $$\frac{0.99 M}{\log x}<\pi(x;q,a)-\pi(x-M;q,a)<\frac{1.01M}{\log x}.$$
    \end{thm}
    \begin{claim} \label{Baker Harman Pintz corollary}
    For $k\ge 2$, if $n$ is large enough then the interval $(n-n^{1-0.42/k},n]$ contains a number of the form $m=(p^k-1)k$ where $p$ is prime and $k|(p-1)$.
    \end{claim}
    \begin{proof}
        Applying \autoref{Baker Harman Pintz} with $a=1$, $q=k$, $x=(n/k+1)^{1/k}$ and $M=x^{0.56}$ gives that for large $n$ there is a prime
        $$p\in\left((n/k+1)^{1/k}-(n/k+1)^{0.56/k},(n/k+1)^{1/k}\right].$$
        Then
        $$p^k\in \left(n/k+1-(n/k+1)^{1-0.43/k},(n/k+1)\right]\Longrightarrow(p^k-1)k\in\left(n-n^{1-0.42/k},n\right].$$
    \end{proof} It is clear than $M_k(S_n)$ is increasing in $n$ since $S_n\subseteq S_{n+1}$. By \autoref{Baker Harman Pintz corollary}, there exists $m\in(n-n^{1-0.42/k},n]$ of the form $m=(p^k-1)k$ for prime $p$ and $k|(p-1)$. Then
$$\begin{aligned}
    M_k(S_n)&\ge (n-n^{1-0.42/k})!^{1/k+O(1/\log(n-n^{1-0.42/k}))}\ge n^{-n^{1-0.42/k}/k+o(1)}(n!)^{1/k+O(1/\log n)}\\
    &\ge(n!)^{1/k+O(1/\log n)}.
\end{aligned}$$
\qed

\section{Conjugacy $S_2$-sets} \label{Conjugacy section}

We will show that \autoref{Sn times Sn result} is a consequence of the following recipe for constructing $S_2[g]$-sets in $\Gamma\times\Gamma$.

\begin{prop} \label{recipe}
    Let $\Gamma$ be a group, $\pi\in\Gamma$, and let $A\subseteq\Gamma$ have the property that for any $\mu\in\Gamma$,
    $$|\{\alpha\in A:\alpha\pi\alpha^{-1}=\mu\}|\le g.$$
    Then $\{(\alpha,\alpha\pi):\alpha\in A\}$ is an $S_2[g]$-set in $\Gamma\times\Gamma$.
\end{prop}
\begin{proof}
    We let $(\mu_1,\mu_2)\in\Gamma\times\Gamma$ and consider the number of pairs $(\alpha,\beta)$ such that $(\alpha,\pi\alpha)(\beta,\pi\beta)=(\mu_1,\mu_2)$. These equations give
    $\alpha\beta=\mu_1$ and $\pi\alpha\pi\beta=\mu_2$. Solving for $\beta$, we obtain $$\alpha^{-1}\mu_1=\beta=\pi^{-1}\alpha^{-1}\pi^{-1}\mu_2$$
    so
    $$\alpha\pi\alpha^{-1}=\pi^{-1}\mu_2\mu_1^{-1}.$$
    By assumption, the number of $\alpha$ satisfying this equation is at most $g$. Since $\alpha,\mu_1,\mu_2$ determine $\beta$, it follows that the number of such pairs $(\alpha,\beta)$ is at most $g$.
\end{proof}

\begin{proof}[Proof of \autoref{Sn times Sn result}]
    Due to the differing sign of odd and even cycles we must consider two cases in order to obtain the results for the alternating group.
    
    \textit{Case 1: $n$ is odd.} Let $\pi$ be a cyclic permutation and let $A=\{\alpha\in S_n:\alpha(1)=1\}.$ Now if $\mu\in S_n$ and $\alpha\pi\alpha^{-1}=\mu$, then $\mu$ must be a cyclic permutation $(m_1\ m_2\ \cdots\  m_n)$, where we choose $m_1=1$. Write $\pi=(p_1\ p_2\ \cdots\  p_n)$, where $p_1=1$. Then
    $$(1\ \alpha(p_2)\ \cdots\ \alpha(p_n))=(\alpha(p_1)\ \alpha(p_2)\ \cdots\ \alpha(p_n))=\alpha\pi\alpha^{-1}=(1\ m_2\ \cdots\ m_n).$$
    But then $\alpha(p_i)=m_i$ for every $i$, and $\alpha$ is determined. By \autoref{recipe}, $\{(\alpha,\alpha\pi):\alpha\in A\}$ is an $S_2$-set and (a) is proved. If we drop the restriction that $\alpha(1)=1$ we are led to the equation
    $$(\alpha(p_1)\ \alpha(p_2)\ \cdots\ \alpha(p_n))=(m_1\ m_2\ \cdots\ m_n).$$
    By cycling the $m_i$, we may assume that $m_1=\alpha(p_1)$. Then $\alpha(p_i)=m_i$ for every $i$. So, the choice of $\alpha(p_1)$ determines the rest of its values, so there are exactly $n$ such $\alpha$. By \autoref{recipe}, $\{(\alpha,\alpha\pi):\alpha\in S_n\}$ is an $S_2[n]$-set and (b) is proved. We now extend the construction to $A_n$. Since $\pi$ is an odd cycle, $\pi\in A_n$. Therefore, if $B=\{\alpha\in A_n:\alpha(1)=1\}$, then $\{(\alpha,\alpha\pi):\alpha\in B\}\subseteq A_n\times A_n$, $\{(\alpha,\alpha\pi):\alpha\in A_n\}\subseteq A_n\times A_n$, and these are clearly an $S_2$-set and an $S_2[n]$-set. We count $|\{(\alpha,\alpha\pi):\alpha\in B\}|=(n-1)!/2$ and $|\{(\alpha,\alpha\pi):\alpha\in A_n\}|=n!/2$, proving (c) and (d).

    \textit{Case 2: $n$ is even.} Let $\pi$ be an $(n-1)$-cycle such that $\pi(1)\ne 1$, and let $A=\{\pi\in S_n:\pi(1)=1\}$. Let $\mu\in S_n$ and suppose that $\alpha\pi\alpha^{-1}=\mu$. Let $\pi=(1\ p_2\ \cdots\ p_{n-1})(p_n)$, and observe that $\mu$ must be of the form $\mu=(m_1\ m_2\  \cdots\ m_{n-1})(m_n)$. So $\alpha\pi\alpha^{-1}=\mu$ gives
    $$(1\ \alpha(p_2)\ \cdots\ \alpha(p_{n-1}))(\alpha(p_n))=(m_1\ m_2\ \cdots\ m_{n-1})(m_n).$$
    This implies $1\in\{m_1,\ldots,m_{n-1}\}$ and $\alpha(p_n)=m_n$. By cycling $m_1,\ldots,m_{n-1}$ we may assume that $m_1=1$, and we see that $\alpha(p_i)=m_i$ for $2\le i\le n-1$. Therefore $\alpha$ is determined, and \autoref{recipe} implies that $\{(\alpha,\alpha\pi):\alpha\in A\}$ is an $S_2$-set, proving (a). If we drop the requirement $\alpha(1)=1$ then we are led to the equation
    $$(\alpha(p_1)\ \alpha(p_2)\ \cdots\ \alpha(p_{n-1}))(\alpha(p_n))=(m_1\ m_2\ \cdots\ m_{n-1})(m_n).$$
    Thus $\alpha(p_n)=m_n$, and $\alpha(p_2),\ldots,\alpha(p_{n-1})$ are determined by the choice of $\alpha(p_1)\in\{m_1,\ldots,m_{n-1}\}$. So \autoref{recipe} gives that $\{(\alpha,\alpha\pi):\alpha\in S_n\}$ is an $S_2[n-1]$-set, proving (b). Now since $\pi$ is an odd cycle, $\pi\in A_n$. Let $B=\{\alpha\in A_n:\alpha(1)=1\}$. Clearly $\{(\alpha,\alpha\pi):\alpha\in B\}$ is an $S_2$-set which similarly to the case where $n$ is odd proves (c). Finally, $\{(\alpha,\alpha\pi):\alpha\in A_n\}$ is an $S_2[n-1]$-set which proves (d).
\end{proof}

We briefly divert to discuss the question of using Sidon sets in $\Gamma$ to find Sidon sets in a subgroup $H\le\Gamma$. If $A\subseteq\Gamma$ is an $S_k'$-set, then so is $\gamma A$ for every $\gamma\in\Gamma$, so taking the average value of $|\gamma A\cap H|$ proves that $M_k'(H)\ge M_k'(\Gamma)/h$, where $h=|G:H|$. However, if $A$ is an $S_k$-set, this translation property does not hold and there are cases where $|M_k(\Gamma)|/|M_k(H)|$ can be arbitrarily large even while $|\Gamma:H|$ is fixed (for example, this occurs in \autoref{Odlyzko result}). Thus, we find it interesting that our construction implies the existence of large $S_2$-sets in certain subgroups $H\times H\subseteq S_n\times S_n$. Above we have shown this when $H=A_n$, but in fact it holds for an arbitrary $H$ which contains $\pi$. Suppose $H\subseteq S_n$ contains the element $\pi$. With $A=\{\alpha\in H:\alpha(1)=1\}$, $B=\{(\alpha,\alpha\pi):\alpha\in A\}$ and $B'=\{(\alpha,\alpha\pi):\alpha\in H\}$, we then have $B,B'\subseteq H\times H$. Since these are subsets of the full constructions, it is clear that $B$ ($B'$) is an $S_2$-set ($S_2[n]$-set) and moreover $|B'|=|H|.$ To find $|B|$, we note that $A$ is the stabilizer subgroup $H_1$, so by the orbit-stabilizer theorem $|A|=|H|/|H\cdot 1|\ge |H|/n$ and therefore $|B|\ge |H|/n$.

We conclude this section by generalizing \autoref{Sn times Sn result} (b) and (d) to any group with a large conjugacy class.

\begin{prop}
Suppose $\Gamma$ is a group with a conjugacy class of size $m$. Then
$$M_{2,g}(\Gamma\times\Gamma)\ge m$$
where $g=|\Gamma|/m$.
\end{prop}
\begin{proof}
    Let $A$ be a conjugacy class of size $m$, and fix $\pi\in A$. For $\mu\in\Gamma$, let $B_\mu=\{\alpha\in A:\alpha\pi\alpha^{-1}=\mu\}$.
    If $\mu\not\in A$ then $B_\mu=\emptyset$. If $\mu\in A$ then there exists $\alpha_0\in A$ such that $\alpha_0\pi\alpha_0^{-1}=\mu$. Now $B_\mu\subseteq\alpha_0\Gamma_\pi$, where $\Gamma_\pi$ is the stabilizer of $\pi$ in the conjugacy action of $\Gamma$. The orbit-stabilizer theorem gives $$|\alpha_0\Gamma_\pi|=\frac{|\Gamma|}{|A|}=\frac{|\Gamma|}{m}.$$
    Apply \autoref{recipe}.
\end{proof}

\section{Probabilistic bounds} \label{probabilistic subsection}

\begin{proof}[Proof of \autoref{Probabilistic constructions}]

(a) Define a hypergraph $H$ where $V(H)=B$ and $e\subseteq B$ whenever there exist $\alpha,\beta,\gamma,\delta\in B$ with $\alpha\beta=\gamma\delta$ and $\{\alpha,\beta,\gamma,\delta\}=e$. We classify edges by the number and position of the distinct elements in the equation $\alpha\beta=\gamma\delta$: with $\alpha,\beta,\gamma,\delta$ being distinct elements, every edge is of one of the following forms:
\begin{itemize}
    \item[(1)] $\alpha^2=\beta^2$
    \item[(2)] $\alpha\beta=\beta\alpha$
    \item[(3)] $\alpha\beta=\gamma\alpha$
    \item[(4)] $\alpha^2=\beta\gamma$
    \item[(5)] $\alpha\beta=\gamma\delta$.
\end{itemize}
By the assumption on $B$, there are no edges of the form (1) or (2). In forms (3), (4), (5) it is possible to solve for $\gamma$ in terms of the other elements. Thus, the number of equations of type (3) or (4) is at most $2b^2$ and the number of equations of type (5) is at most $b^3$. To bound the independence number of $H$ we borrow from \cite{babai1985sidon} the following non-uniform version of Tur\'an's theorem.
\begin{prop}[Babai-S\'os \cite{babai1985sidon}] \label{Non-uniform Turans theorem}
Let $e_r$ denote the number of edges of size $r$ in the hypergraph $H$ with $n$ vertices. Let
$$f(k)=\sum_re_r{k\choose r}/{n\choose r}.$$
Then
$$\alpha(H)\ge\max\{k-f(k):1\le k\le n\}.$$
\end{prop}
In the setup above, choosing $k=(0.49b)^{1/3}$ gives for large enough $b$
$$\begin{aligned}
    \frac{f(k)}{k}\le\frac{2b^2{k\choose 3}}{{b\choose 3}k}+\frac{b^3{k\choose 4}}{{b\choose 4}k}
    =\left(\frac{2k^2}{b}+\frac{k^3}{b}\right)(1+o(1))<\frac{1}{2}.
\end{aligned}$$
Thus, $M_2(\Gamma)\ge\alpha(H)\ge k-f(k)>k/2>(0.39+o(1))b^{1/3}.$

(b) Let $n=|\Gamma|$, let $I$ be the set of $i$ involutions in $\Gamma$, and define a hypergraph $H$ with $V(H)=\Gamma$ and $e\in E(H)$ whenever there exist $\alpha,\beta,\gamma,\delta\in \Gamma$ with $\alpha\ne\beta\ne\gamma\ne\delta$, $\alpha\beta^{-1}\gamma\delta^{-1}=1$, and $\{\alpha,\beta,\gamma,\delta\}=e$. For distinct $\alpha,\beta,\gamma,\delta$, the edges appear in the following forms.
\begin{itemize}
    \item[(1)] $\alpha\beta^{-1}\alpha\beta^{-1}=1$.
    \item[(2)] $\alpha\beta^{-1}\alpha\delta^{-1}=1$
    \item[(3)] $\alpha\beta^{-1}\gamma\delta^{-1}=1.$
\end{itemize}
These possibilities are exhaustive up to permuting the symbols, since $\alpha\beta^{-1}\gamma\beta^{-1}=1\Longrightarrow\beta\gamma^{-1}\beta\alpha^{-1}=1$ which is a type (2) equation and because $\alpha\beta^{-1} \delta \alpha^{-1} = 1$ implies $\beta = \delta$. Now (1) holds if and only if $\alpha\in I\beta$ where $I$ is the set of involutions of $\Gamma$, so the number of equations in form (1) is $ni$. In forms (2) and (3) one can solve for $\beta$ in terms of the other elements, so there are at most $n^2$ equations in form (2) and $n^3$ equations in form (3). We have
$$\begin{aligned}
    \frac{f(k)}{k}&\le\frac{ni{k\choose 2}}{{n\choose 2}k}+\frac{n^2{k\choose 3}}{{n\choose 3}k}+\frac{n^3{k\choose 4}}{{n\choose 4}k}=\left(\frac{ki}{n}+\frac{k^2}{n}+\frac{k^3}{n}\right)(1+o(1)).
\end{aligned}$$
If $i=o(n^{2/3})$ then choosing $k=(0.49n)^{1/3}$ gives $f(k)/k<1/2$ for large $n$ and we have $M_2'(\Gamma)=\alpha(H)>(0.39+o(1))n^{1/3}$. If $i\ge Cn^{2/3}$ then choosing $k=n/((4/C+4)i)$ implies $k\le n^{1/3}/4$ so for large $n$ we have
$$\frac{f(k)}{k}<\frac{1}{4}+o(1)+\frac{1}{64}<\frac{1}{2}$$
and so $M_2'(\Gamma)\ge k-f(k)>k/2=\Omega(n/i).$
\end{proof}

In the proofs above, we counted 5 distinct forms of the forbidden equation for an $S_2$-set and 3 forms for an $S_2'$-set. As $k$ increases, the number of distinct forms also increases. Thus, we expect that probabilistic bounds for $k\ge 3$ would be considerably more difficult to apply.

\section{Upper bounds} \label{Upper bound section}

\begin{prop}
    If $k\ge 2$ be fixed. If $\Gamma$ contains an abelian subgroup of index 2, then
$$M_k(\Gamma)\le(1/2^{1/k}+o(1))|\Gamma|^{1/k}.$$
    where $o(1)\to 0$ as $|\Gamma|\to\infty$.
\end{prop}
\begin{proof}
Suppose $\Gamma$ has an abelian subgroup $H$ of index 2. Let $A\subseteq\Gamma$ be an $S_k$-set. Since all but 1 of the elements of $A$ must belong to $\Gamma-H$ and all $k$-letter words taken from $\Gamma-H$ have a product which belongs to the same coset, we obtain
$(|A|-1)^k\le\frac{|\Gamma|}{2}$ so
$$M_k(\Gamma)\le(1/2^{1/k}+o(1))\gamma^{1/k}.$$
\end{proof}

Next we consider the case of fixed index $h\ge 3$. Before proving our main result we need some lemmas about certain real-valued vectors indexed by a group. These are essentially fractional/stability versions of some lemmas in Dimovsky's proof \cite{dimovski1992groups} that $M_k(\Gamma)<|\Gamma|^{1/k}$ when $|\Gamma|>1$, and we refer the reader to that paper for the full setup required to prove \autoref{dimovski generalization}. 

\begin{lemma} \label{dimovski generalization}
    Suppose $K$ is a group of order $h$, and $x\in\mathbb R^K$ is a vector with the property
    $$\forall g\in K\ \sum_{k\in K}x_kx_{k^{-1}g}=\frac{1}{h}.$$
    Then $x_1=1/h$.
\end{lemma}
\begin{proof}
    Parts (b)-(d) of the proof of Theorem 1 in \cite{dimovski1992groups} still hold (we do not need part (a)). In the setup of part (e), let $x=\sum_{g\in K}x_gg=A_1+\cdots+A_s$. We have
    $$x\cdot x=\sum_{g\in K}\left(\sum_{k\in K}x_kx_{k^{-1}g}\right)g=\frac{1}{h}\sum_{g\in K}g=e_1.$$
    On the other hand, $x\cdot x=(A_1+\cdots+A_s)^2=A_1^2+\cdots+A_s^2$. Thus $A_1^2=1$ so $
A_1=1$ and $A_t^2=0$ for $t\ge 2$. Since the trace of a nilpotent matrix is 0, we have
    $$hx_1=\sum_{g\in K}x_g\chi(g)=\chi(x)=\sum_{i=1}^sf_i\mathrm{Tr}(A_i)=A_1=1$$
    since $f_1=1$ and $A_1$ is a $1\times 1$ matrix over $\mathbb C$. So, $x_1=1/h$.
\end{proof}

\begin{lemma} \label{dimovski generalization approximate}
    Suppose $K$ is a group of order $h$. For any $\varepsilon>0$, there exists $\delta>0$ such that any $x\in[0,1]^K$ with the property
    $$\forall g\ \left|\sum_{k\in K}x_kx_{k^{-1}g}-\frac{1}{h}\right|\le\delta$$
    satisfies $x_1>1/h-\varepsilon$.
\end{lemma}
\begin{proof}
    If not, then there is some $\varepsilon>0$ and a sequence of vectors $x^{(n)}$ such that
    $$\left(\sum_{k\in K}x^{(n)}_kx^{(n)}_{k^{-1}g}\right)_{g\in K}\to\mathbf 1/h$$
    as $n\to\infty$, while $x^{(n)}_1\le 1/h-\varepsilon$. Since $[0,1]^K$ is compact, by taking subsequences we may assume that $x^{(n)}$ converges to some $x\in[0,1]^K$. Since the functions
    $$y\mapsto\left(\sum_{k\in K}y_ky_{k^{-1}g}\right)_{g\in K}\text{ and }y\mapsto y_1$$
    are continuous, we have $\forall g\in K\ \sum_{k\in K}x_kx_{k^{-1}g}=1/h$ and $x_1\le 1/h-\varepsilon$. But by \autoref{dimovski generalization}, this is impossible.
\end{proof}

We are now ready to prove our upper bound. The proof still closely follows \cite{dimovski1992groups}.

\begin{proof}[Proof of \autoref{Sk upper bound}]
    Let $k=2r$, $A$ be an $S_k$-set in $\Gamma$ with $|A|>(1-\varepsilon)|\Gamma|^{1/k}$, and $K=\Gamma/H$. Define $L=\{\alpha_1\cdots\alpha_r:\alpha_i\in A\}$. Then $|L|=|A|^r\ge(1-r\varepsilon)|\Gamma|^{1/2}$, and $L$ is an $S_2$-set in $\Gamma$. Let $x_g=|L\cap g|/\sqrt{|\Gamma|}$ for cosets $g\in K$. Since $L$ is an $S_2$-set, the products $\alpha\beta$ for $\alpha,\beta\in L$ are all distinct and cover at least $(1-2r\varepsilon)|\Gamma|$ elements of $\Gamma$. By counting $\{\alpha\beta:\alpha\beta\in g,\alpha,\beta\in L\}$ it follows that
    $$\forall g\in K\ \frac{1}{h}-2r\varepsilon=\frac{|\Gamma|/h-2r\varepsilon|\Gamma|}{|\Gamma|}\le\sum_{k\in K}x_kx_{k^{-1}g}\le\frac{|\Gamma|/h}{|\Gamma|}=\frac{1}{h}.$$
    By \autoref{dimovski generalization approximate}, we can choose $\varepsilon$ small enough (by minimizing the choice of $\varepsilon$ over all finite groups of order $h$) so that this implies $x_1\ge1/(2h)$, i.e. $|L\cap 1|\ge\sqrt{|\Gamma|}/(2h)$. If $|\Gamma|$ is large enough, then $\sqrt{|\Gamma|}/(2h)\ge 2$, contradicting that $H$ is abelian.
\end{proof}

If more specific information about $\Gamma/H$ is known we can sometimes obtain better bounds.

\begin{prop} \label{Quotient has exponent 2}
Suppose $\Gamma$ has a normal abelian subgroup $H$ and $\Gamma/H\simeq\mathbb Z_2^d$. Then
$$M_2(\Gamma)\le((1-1/2^d)^{1/2}+o(1))|\Gamma|^{1/2}$$
\end{prop}
\begin{proof}
    Suppose $A$ is an $S_2$-set in $\Gamma$. Let the cosets of $H$ be $H=\beta_1H,\ldots,\beta_{2^d}H$. Let $x_i=|A\cap\beta_iH|$. Since $(A\cap\beta_iH)^2\subseteq H$ and $|A\cap H|\le 1$, we have
    $$\begin{aligned}
       \frac{|A|^2}{2^d-1}+O(|A|)\le 1+\sum_{i\ne 1}\left(\frac{|A|-1}{2^d-1}\right)^2\le\sum_{i}x_i^2\le|\Gamma|/2^d
    \end{aligned}$$
    and the claim follows.
\end{proof}

It seems that other bounds could be proven on an ad-hoc basis depending on the structure of $\Gamma/H$. We now turn to $S_k'$-sets. Here, if $k=2$ then the existence of abelian subgroups tells us nothing because large $S_2'$-sets exist (they are precisely the Sidon sets). When $k\ge 3$, the situation is different.

\begin{prop} \label{Sk' upper bound}
    If $\Gamma$ contains an abelian subgroup of index $h$, then for any $k\ge 3$ we have
    $$M_k'(\Gamma)\le h(k-1).$$
\end{prop}
\begin{proof}
 Let $A\subseteq\Gamma$ be an $S_k'$-set and suppose that $H\le\Gamma$ is a subgroup of index $h$. Then $|A\cap H|\le k-1$. For suppose there existed distinct $\alpha_1,\ldots,\alpha_k\in A\cap H$. Then we have
$$\alpha_1\alpha_k^{-1}\alpha_2\alpha_1^{-1}\alpha_3\alpha_2^{-1}\cdots\alpha_k\alpha_{k-1}^{-1}=1$$
while no element appears next to its inverse in the above equation, contradicting the definition. Moreover, for any $\gamma\in\Gamma$ we have $\gamma A$ is also an $S_k'$-set:
$$(\gamma\alpha_1)(\gamma\beta_1)^{-1}\cdots(\gamma\alpha_k)(\gamma\beta_k)^{-1}=1\Longrightarrow\gamma\alpha_1\beta_1^{-1}\cdots\alpha_k\beta_k^{-1}\gamma^{-1}=1\Longrightarrow\alpha_1\beta_1^{-1}\cdots\alpha_k\beta_k^{-1}=1$$
and $\gamma\alpha_i\ne\gamma\beta_i\ne\gamma\alpha_{i+1}$ implies $\alpha_i\ne\beta_i\ne\alpha_{i+1}.$ Therefore, $|\gamma A\cap H|\le k-1$ for every $\gamma\in\Gamma$. Thus:
$$\begin{aligned}
    |A||H|=\sum_{\alpha\in A}|\{\gamma\in\Gamma:\gamma\alpha\in H\}|=\sum_{\gamma\in\Gamma}|\gamma A\cap H|\le|\Gamma|(k-1)
\end{aligned}$$
and so $|A|\le(k-1)|\Gamma|/|H|=h(k-1).$ 
\end{proof}

This means that for $k\ge 3$, large $S_k'$-sets can only exist in groups which have no abelian subgroups of bounded index. 

The following bound is very easy but could be useful for ruling out $S_k$-sets in certain groups.


\begin{prop}
    Let $m_k(\Gamma)$ be the number of $\ell$, $2\le\ell\le k$, for which $\Gamma$ contains an element of order $\ell$, and let $n_k(\Gamma)$ be the number of elements of order larger than $k$. Then
    $$M_k(\Gamma)\le m_k(\Gamma)+n_k(\Gamma).$$
    unless $|\Gamma|=1$.
\end{prop}
\begin{proof}
    Let $A$ be an $S_k$-set in $\Gamma$. Let $A_m=\{\alpha\in A:2\le o(\alpha)\le k\}$ and $A_n=\{\alpha\in A:o(\alpha)>k\}$. If $|A|=1$ then the conclusion is immediate. Otherwise, $1\not\in A$ implying $A=A_m\sqcup A_n$. Since $A$ is an $S_\ell$-set for every $2\le\ell\le k$, $A$ has at most one element of every order between $2$ and $k$; thus $|A_m|\le m_k(\Gamma)$. Now $|A_n|\le n_k(\Gamma)$ is true by definition so the result follows.
\end{proof}

\section{Extremal problems for directed graphs}

In this section, we prove Theorems \ref{Cll girth upper bound} and \ref{Cll lower bound} and Corollary \ref{Path separation application}.  As is common, we may use in our constructions some divisibility or prime factor conditions.  However, it is not clear that the functions $m^+(n, \mathcal{F})$, $m^-(n, \mathcal{F})$, $m^0(n, \mathcal{F})$ are monotone, and hence we cannot simply remove a small number of vertices to obtain lower bounds without additionally checking the degrees. We will therefore need the following lemma.

\begin{lemma} \label{density lemma}
Let $\varepsilon,a>0$ and suppose $G$ is a directed graph on $n$ vertices in which $\delta^+(G),\delta^-(G)\ge n^a$. Let $m\in[n/2,n]$ be an integer, and let $G'$ be obtained from $G$ by randomly deleting each vertex independently with probability $p=1-m/n$. Then with positive probability, $\delta^+(G')\ge(1-\varepsilon)(1-p)\delta^+(G)$, $\delta^-(G')\ge (1-\varepsilon)(1-p)\delta^-(G)$, and $|V(G')|=m$ all occur for large enough $n$.
\end{lemma}
\begin{proof}
    First we note that $|V(G')|\sim\mathrm{Bin}(n,1-p)$ and its expected value is $m$. By the standard central limit theorem we have that $\mathbb{P}(|V(G')| =m) = \Omega\left( \frac{1}{n}\right)$. 
    Now for any vertex $v\in V(G')$, we have $\delta^+_{G'}(v)\sim\mathrm{Bin}(\delta^+_G(v),p)$. So, the Chernoff bound \cite{chernoff1952measure} gives
    $$\begin{aligned}
        \mathbb P[\delta^+_{G'}(v)<(1-\varepsilon)(1-p)\delta^+_G(v)]\le e^{-\varepsilon^2(1-p)\delta_G^+(v)/2}\le e^{-\varepsilon^2(1-p)n^a/2}
    \end{aligned}$$
    and similarly for $\delta_{G'}^-(v)$. Thus, the probability that there exists $v\in V(G')$ with either $\delta_{G'}^+(v)<(1-\varepsilon)(1-p)\delta_G^+(v)$ or $\delta_{G'}^-(v)<(1-\varepsilon)(1-p)\delta_{G}^-(v)$ is at most
    $$2ne^{-\varepsilon^2(1-p)n^a/2}\ll1/n.$$
\end{proof}

\subsection{Proof of \autoref{Cll girth upper bound}}
    We begin with the first inequality. For the time being suppose that $n=(p^k-1)k$ where $p$ is a prime with $k|(p-1)$. By \autoref{Odlyzko result}, there exists a group $\Gamma$ and an $S_k$-set $A\subseteq\Gamma$ with $|\Gamma|=(p^k-1)k$ and $|A|=(p-1)/k=k^{-1-1/k}n^{1/k}+O(1)$. Let $G=\mathrm{Cay}(\Gamma,A)$ (this is a directed Cayley graph with no loops or opposite edges), i.e. $(\alpha,\beta)\in E(G)\Longleftrightarrow\alpha^{-1}\beta\in A$. Note that every $v\in V(G)$ has $d^+(v)=d^-(v)=|A|$. A directed walk of length $k$ in $G$ is a sequence $\alpha,\alpha\beta_1,\alpha\beta_1\beta_2,\ldots,\alpha\beta_1\cdots\beta_k$, where $\alpha\in\Gamma$ and $\beta_i\in A$. If two such walks $\alpha,\ldots,\alpha\beta_1\cdots\beta_k$ and $\alpha',\ldots,\alpha'\beta_1'\cdots\beta_k'$ have the same initial and same terminal vertices, then we have $\alpha=\alpha'$ and $\alpha\beta_1\cdots\beta_k=\alpha'\beta_1'\cdots\beta_k'$, thus $\beta_1\cdots\beta_k=\beta_1'\cdots\beta_k'$. Since $A$ is an $S_k$-set, we have $(\beta_1,\ldots,\beta_k)=(\beta_1',\ldots,\beta_k')$ and the walks are the same. So, $m^0(n,\mathcal F_k)\ge k^{-1-1/k}n^{1/k}+O(1)$ for such $n$. 

    Now let $n\in\mathbb N$ be arbitrary. Let $G$ be the graph on $m=(p^k-1)k$ vertices considered above. Using \autoref{Baker Harman Pintz corollary}, we may choose $(1-o(1))m \leq n \leq m$ and by applying Lemma \ref{density lemma}, we have that
    $$\begin{aligned}
       m^0(n,\mathcal F_k)\ge\left(\frac{1}{k^{1+1/k}}-o(1)\right)n^{1/k}.
    \end{aligned}$$

    For the second inequality we first note that $m^0(n,\mathcal F_k)\le m^+(n,\mathcal F_k)$. If a graph with minimum outdegree $\delta^+\ge 1$ contains some $C_{\ell,\ell}$ with $2\le\ell\le k$, say formed by the directed paths $x_0,\cdots, x_{\ell}$ and $y_0,\cdots, y_\ell$ (where $x_0=y_0$ and $x_\ell=y_\ell$) then there exists some directed walk $z_\ell=x_\ell,z_{\ell+1},\ldots,z_k$. Then $x_0,\ldots,x_{\ell},z_{\ell+1},\ldots,z_{k}$ and $y_0,\ldots,y_\ell,z_{\ell+1},\ldots,z_k$ form a graph in $\mathcal F_k$ (we will use this fact of `extending the walks' frequently below). Thus $m^+(n,\mathcal F_k)\le m^+(n,C_{\ell,\ell})$. 

    Next we consider the third inequality. Let $G$ be an $n$-vertex $\mathcal C_{k,k}$-free directed graph with minimum degree $\delta^+$. We first remove short cycles of type $\ne 0$ from $G$, by applying the following lemma which will also be useful later.
    \begin{lemma} \label{Removing short unbalanced cycles}
        Let $h\in\mathbb N$, $\varepsilon>0$. Suppose $n$ is large enough and $\delta^+\gg\log n$. Let $G$ be an $n$-vertex digraph with $\delta^+(G)\ge\delta^+$. Then $G$ has a spanning subgraph $G'$ with $\delta^+(G')\ge\frac{1-\varepsilon}{2h}\delta^+$ in which every closed walk of length at most $2h-1$ has type $0$.
    \end{lemma}
    \begin{proof}
        Randomly partition the vertices of $G$ as $V(G)=V_0\sqcup\cdots\sqcup V_{2h-1}$ so that each vertex $v$ is assigned to one part $P(v)$, uniformly and independently, and let $G'$ be the graph obtained by keeping only the edges from $V_i$ to $V_{i+1}\pmod{2h}$. For each $v\in V(G)$ we have that $d_{G'}^+(v)=\sum_{w:(v,w)\in E(G)}\textbf{1}_{P(w)=P(v)+1}$, where the $\textbf{1}_{P(w)=P(v)+1}$ are $d_G^+(v)$ independent Bernoulli random variables with parameter $1/(2h).$ The Chernoff bound \cite{chernoff1952measure} gives
    $$\mathbb P\left[d_{G'}^+(v)<\frac{1-\varepsilon}{2h}d_G^+(v)\right]\le e^{-\frac{\varepsilon^2d_G^+(v)}{4h}}\le e^{-\frac{\varepsilon^2\delta^+}{4h}}.$$
    Therefore 
    $$\mathbb P\left[\exists v\ d_{G'}^+(v)<\frac{1-\varepsilon}{2h}d_G^+(v)\right]\le ne^{\frac{-\varepsilon^2\delta^+}{4h}}<1.$$
    Thus, with positive probability $d_{G'}^+(v)\ge\frac{1-2\varepsilon}{2h}\delta^+$ for every $v$. Now the definition of $G'$ guarantees that every cycle in $G'$ of length at most $2h-1$ has type $0$.
    \end{proof}
    Consider the graph $G'$ obtained from $G$ by \autoref{Removing short unbalanced cycles}, with $h=k$. We claim that $G'$ is $\mathcal F_k$-free. For if $x_0,\ldots,x_k$, $y_0,\ldots,y_k$ are two walks with the same initial and same terminal vertices, there exists a minimum $i$ such that $x_i\ne y_i$. If $x_i=y_{i'}$ for some $i'\ne i$, then $x_0,\ldots,x_i=y_{i'},y_{i'-1},\ldots,y_0=x_0$ is an unbalanced closed walk of length at most $2k-1$, contradicting the definition of $G'$. So, $x_i\not\in\{y_0,\ldots,y_k\}$. There exists a minimum $j>i$ such that $x_j\in\{y_0,\ldots,y_k\}$. Let $x_j=y_{j'}$. If $j=j'$ then $x_i,\ldots,x_j,y_{j-1},\ldots,y_{i-1}$ is a $C_{j-i+1,j-i+1}$ where $2\le j-i+1\le k$, a contradiction. If $j\ne j'$ then $j<k$ or $j'<k$ and so $x_0,\ldots,x_j,y_{j'-1},\ldots,y_0$ is an unbalanced closed walk of length at most $2k-1$, a contradiction.
    Thus 
    $$m^+(\mathcal F_k)\ge\frac{1-\varepsilon}{2k}\delta^+(G)$$
    and the inequality follows.

    We finally turn to the fourth inequality. Let $G$ be an $\mathcal F_k$-free graph with minimum outdegree $\delta^+\ge 1$. Let $v\in V$. Let $L_i$ be the set of vertices $x$ for which there exists a directed walk of length $i$ from $v$ to $x$. If $i<k$ and there exist distinct directed walks of length $i$ in $G$ with the same initial and same terminal vertices, then the condition $\delta^+\ge 1$ allows us to extend the walks to length $k$ while still having the same terminal vertices. Thus, $G$ is also $\mathcal F_i$-free for $i\le k$. Hence, for $x,y\in L_i$ with $i\le k-1$ we have $N^+(x)\cap N^+(y)=\emptyset$, and so $|L_{i+1}|\ge\delta^+|L_i|$. It follows that
    $$n\ge|L_k|\ge(\delta^+)^k.$$

\qed

\subsection{Proof of \autoref{Cll lower bound}}
    We begin by defining the graphs that will prove the first inequality. 

    \begin{defn} \label{Glm definition}
        Let $\ell,m\in\mathbb N$. We define a graph $G=G_{\ell,m}$ on a vertex set $V=V_0\sqcup\cdots\sqcup V_{\ell-2}\sqcup W_0\sqcup\cdots\sqcup W_{\ell-2}$, where for each $i$ we have $V_i=\{v_{ijk}:j,k\in[m]\}$ and $W_i=\{w_{ijk}:j,k\in[m]\}$. Let $V_{ij}=\{v_{ij1},\ldots,v_{ijm}\}$ and $W_{ij}=\{w_{ij1},\ldots,w_{ijm}\}$. The edges of $G$ are defined as follows: for $0\le i\le\ell-3$ and $j,k\in[m]$ let 

    \begin{equation}\label{graph equation}\begin{aligned}\forall j'&\ (v_{ijk},v_{(i+1)j'k})\in E(G),\\
    &\ (v_{(\ell-2)jk}, w_{0j'k}) \in E(G);\\
    \forall k'&\ (w_{ijk}, w_{(i+1)jk'})\in E(G),\\
    &\ (w_{(\ell-2)jk}, v_{0jk'}) \in E(G) .\end{aligned}\end{equation}
        \end{defn}

\begin{figure}[h]
\begin{center}

\begin{tikzpicture}[scale=0.5]

\filldraw [black] (-4,4.5) circle (0.2);
\filldraw [black] (-4,3.5) circle (0.2);
\filldraw [black] (-4,0.5) circle (0.2);
\filldraw [black] (-4,-0.5) circle (0.2);
\filldraw [black] (0,4.5) circle (0.2);
\filldraw [black] (0,3.5) circle (0.2);
\filldraw [black] (0,0.5) circle (0.2);
\filldraw [black] (0,-0.5) circle (0.2);
\filldraw [black] (4,4.5) circle (0.2);
\filldraw [black] (4,3.5) circle (0.2);
\filldraw [black] (4,0.5) circle (0.2);
\filldraw [black] (4,-0.5) circle (0.2);
\filldraw [black] (8,4.5) circle (0.2);
\filldraw [black] (8,3.5) circle (0.2);
\filldraw [black] (8,0.5) circle (0.2);
\filldraw [black] (8,-0.5) circle (0.2);
\filldraw [black] (12,4.5) circle (0.2);
\filldraw [black] (12,3.5) circle (0.2);
\filldraw [black] (12,0.5) circle (0.2);
\filldraw [black] (12,-0.5) circle (0.2);
\filldraw [black] (16,4.5) circle (0.2);
\filldraw [black] (16,3.5) circle (0.2);
\filldraw [black] (16,0.5) circle (0.2);
\filldraw [black] (16,-0.5) circle (0.2);

\draw [->] (-4,4.5) -- (-0.2,4.5);
\draw [->] (-4,4.5) -- (-0.2,3.5);
\draw [->] (-4,3.5) -- (-0.2,4.5);
\draw [->] (-4,3.5) -- (-0.2,3.5);
\draw [->] (-4,0.5) -- (-0.2,0.5);
\draw [->] (-4,0.5) -- (-0.2,-0.5);
\draw [->] (-4,-0.5) -- (-0.2,0.5);
\draw [->] (-4,-0.5) -- (-0.2,-0.5);

\draw [->] (0,4.5) -- (4-0.2,4.5);
\draw [->] (0,4.5) -- (4-0.1,0.6);
\draw [->] (0,3.5) -- (4-0.2,3.5);
\draw [->] (0,3.5) -- (4-0.1,-0.4);
\draw [->] (0,0.5) -- (4-0.1,4.4);
\draw [->] (0,0.5) -- (4-0.2,0.5);
\draw [->] (0,-0.5) -- (4-0.1,3.4);
\draw [->] (0,-0.5) -- (4-0.2,-0.5);

\draw [->] (4,4.5) -- (8-0.2,4.5);
\draw [->] (4,4.5) -- (8-0.1,0.6);
\draw [->] (4,3.5) -- (8-0.2,3.5);
\draw [->] (4,3.5) -- (8-0.1,-0.4);
\draw [->] (4,0.5) -- (8-0.1,4.4);
\draw [->] (4,0.5) -- (8-0.2,0.5);
\draw [->] (4,-0.5) -- (8-0.1,3.4);
\draw [->] (4,-0.5) -- (8-0.2,-0.5);

\draw [->] (8,4.5) -- (12-0.2,4.5);
\draw [->] (8,4.5) -- (12-0.1,0.6);
\draw [->] (8,3.5) -- (12-0.2,3.5);
\draw [->] (8,3.5) -- (12-0.1,-0.4);
\draw [->] (8,0.5) -- (12-0.1,4.4);
\draw [->] (8,0.5) -- (12-0.2,0.5);
\draw [->] (8,-0.5) -- (12-0.1,3.4);
\draw [->] (8,-0.5) -- (12-0.2,-0.5);

\draw [->] (12,4.5) -- (16-0.2,4.5);
\draw [->] (12,4.5) -- (16-0.2,3.5);
\draw [->] (12,3.5) -- (16-0.2,4.5);
\draw [->] (12,3.5) -- (16-0.2,3.5);
\draw [->] (12,0.5) -- (16-0.2,0.5);
\draw [->] (12,0.5) -- (16-0.2,-0.5);
\draw [->] (12,-0.5) -- (16-0.2,0.5);
\draw [->] (12,-0.5) -- (16-0.2,-0.5);

\draw (16,4.5) -- (18,4.5);
\draw (16,4.5) -- (18,4);
\draw (16,3.5) -- (18,4);
\draw (16,3.5) -- (18,3.5);
\draw (16,0.5) -- (18,0.5);
\draw (16,0.5) -- (18,0);
\draw (16,-0.5) -- (18,0);
\draw (16,-0.5) -- (18,-0.5);

\draw [->] (-6,4.5) -- (-4.2,4.5);
\draw [->] (-6,4) -- (-4.1,3.5);
\draw [->] (-6,4) -- (-4.2,4.5);
\draw [->] (-6,3.5) -- (-4.1,3.5);
\draw [->] (-6,0.5) -- (-4.2,0.5);
\draw [->] (-6,0) -- (-4.1,-0.5);
\draw [->] (-6,0) -- (-4.2,0.5);
\draw [->] (-6,-0.5) -- (-4.1,-0.5);

\draw (-4,-1.5) node{$W_2$};
\draw (-0,-1.5) node{$V_0$};
\draw (4,-1.5) node{$V_1$};
\draw (8,-1.5) node{$V_2$};
\draw (12,-1.5) node{$W_0$};
\draw (16,-1.5) node{$W_1$};

\end{tikzpicture}
\caption{The graph $G_{\ell,m}$ when $\ell=4$ and $m=2$.}
\end{center}
\end{figure}
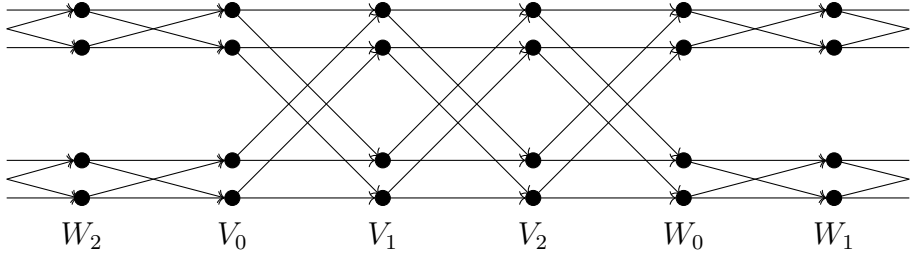

Assume for a contradiction that $G$ contains a $C_{\ell,\ell}$ composed of the two directed paths $x_0,\ldots,x_\ell$ and $y_0,\ldots,y_{\ell}$ where $x_0=y_0$ and $x_\ell=y_\ell$. By the symmetry of the $j$- and $k$-coordinates, we may assume that $x_0\in V_i$ for some $0\le i\le \ell-2$. Then $V(C_{\ell,\ell})\cap W_0=\{x_{\ell-1-i},y_{\ell-1-i}\}.$ Since $x_0=y_0$ and $x_{\ell-1-i}\ne y_{\ell-1-i}$, the structure of $G[V_0\cup\cdots V_{\ell-2}\cup W_0]$ guarantees that $x_{\ell-1-i}=w_{0jk}$ and $y_{\ell-1-i}=w_{0j'k}$ for some $j\ne j'$. Now $x_\ell\in W_{i+1}$ (with the convention $W_{\ell-1}=V_0$). Then the structure of $G[W_0,\ldots,W_{\ell-1}]$ implies that $x_\ell=w_{(i+1)jk'}$ and $y_\ell=w_{(i+1)j'k''}$ for some $k,k''$; but this contradicts $x_\ell=y_\ell$. 

Thus $G_{\ell,m}$ is a $C_{\ell,\ell}$-free digraph on $(2\ell-2)m^2$ vertices with minimum indegree and minimum outdegree $m$. This proves that for every $m$, $m^0((2\ell-2)m^2,C_{\ell,\ell})\ge m$. Given any $n\in\mathbb N$, there is some $n'$ of the form $n'=(2\ell-2)m^2$ with $n'\in(n,n +o(n))$. By applying Lemma \ref{density lemma}, we obtain
$$m^0(n,C_{\ell,\ell})\ge(1-o(1))\left\lfloor \left(\frac{n'}{2\ell-2}\right)^{1/2}\right\rfloor\ge\left(\frac{1}{(2\ell-2)^{1/2}}-o(1)\right)n^{1/2}.$$

The second inequality is immediate.

We now turn to the third inequality. Let $G$ be an $n$-vertex $C_{\ell,\ell}$-free graph with minimum outdegree $\delta^+$. Using \autoref{Removing short unbalanced cycles}, we pass to the directed graph $G'$ with $d:=\delta^+(G')\ge\frac{1-\varepsilon}{2\ell}\delta^+$ in which every closed walk of length at most $2\ell-1$ has type 0. Let $v\in V(G')$, and note there is a set $L_1$ of $d$ vertices in $N_{G'}^+(v)$. Assume we have constructed a set $L_i$ ($i\le\ell-2)$ of $d$ vertices such that
    for any $x,y\in L_i$ there are paths $P_1,P_2$ on $i$ edges oriented from $v$ to $x,y$ respectively so that 
    \begin{equation} \label{Path condition}
    V(P_1),V(P_2)\subseteq\{v\}\cup L_1\cup\cdots\cup L_{i}\text{ and }V(P_1)\cap V(P_2)=\{v\}.
    \end{equation}
    For $x\in L_i$ we have $|N^+(x)|\ge d$ so we can greedily choose distinct vertices $L_{i+1}=\{f(x):x\in L_i\}$ such that $(x,f(x))\in E(G')$ for all $x\in L_i$. Moreover we have $f(x)\not\in \{v\}\cup L_1\cup\cdots\cup L_i$ or else $G'$ would contain a cycle $C$ of length at most $2i+1$ with $t(C)\ne 0$, a contradiction. Thus, for $x,y\in L_i$ we can extend the paths $P_1$ and $P_2$ by the edges $(x,f(x)),(y,f(y))$ to satisfy \autoref{Path condition}. We arrive by induction at the set $L_{\ell-1}$. If there exist $x,y\in L_{\ell-1}$ and $z\in V(G')$ such that $(x,z),(y,z)\in E(G')$, then similarly to the above we have $z\not\in\{v\}\cup L_1\cup\cdots\cup L_{\ell-1}$. Thus, applying \autoref{Path condition} to the vertices $x,y$ and extending the paths by $xz,yz$ gives a copy of $C_{\ell,\ell}$, a contradiction. Hence $N^+_{G'}(x)\cap N_{G'}^+(y)=\emptyset$ so 
    $$n\ge\left|\bigcup_{x\in L_{\ell-1}}N_{G'}^+(x)\right|\ge d\cdot d$$
    which gives 
    $\left(\frac{1-\varepsilon}{2\ell}\delta^+\right)^2\le n$ and the result follows.
\qed

\subsection{Proof of \autoref{Path separation application}}

    Let $\ell=2r$. Assume for the time being that $(2r-2)(2r+1)|n$. First we count Hamilton cycles in the graph $G=G_{r,m}$ from \autoref{Glm definition} with $m^2=n/(2r-2)$. Note that $G[V_0,\ldots,V_{\ell-1}]$ and $G[W_0,\ldots,W_{\ell-1}]$ are each $m$ disjoint copies of a blowup of a directed $P_{\ell-1}$. Let $X_1,\ldots,X_m$ be the components of $G[V_0,\ldots,V_{\ell-1}]$ and let $Y_1,\ldots,Y_m$ be the components of $G[W_0,\ldots,W_{\ell-1}]$.   A \textit{transition vector} is a word
$$t=X_{f(1)}Y_{g(1)}X_{f(2)}Y_{g(2)}\ldots X_{f(m^2)}Y_{g(m^2)}$$
with properties
\begin{itemize}
    \item $f,g:[m^2]\to[m]$
    \item for each $i,j\in[m]$, the contiguous subwords $X_iY_j$ and $Y_jX_i$ each occur exactly once (we consider the vector cyclically, so that $Y_{f(m^2)}X_{f(1)}$ is a contiguous subword).
    \item $f(1)=g(1)=1$.
\end{itemize}
(The importance of the second property comes from the fact that, for any $X_i$ and $Y_j$, there are exactly two vertices in $X_i\cap Y_j$, one in $V_0$ and one in $W_0$. As we will see below, the second property is used to guarantee that certain walks associated with the transition vector visit each vertex in $V_0\cup W_0$ exactly once.) A transition vector is equivalent to an Eulerian circuit in the bidirected $K_{m,m}$. To enumerate Eulerian circuits we refer to the famous result of de Bruijn, van Aardenne-Ehrenfest, Smith, and Tutte.
\begin{thm}[BEST \cite{van1987circuits}]
Let $G$ be a strongly connected digraph in which every vertex $v$ has $d^+(v)=d^-(v)$. Let $t_v(G)$ denote the number of oriented spanning subtrees with root $v$. Then for any $v\in V(G)$ the number of Eulerian circuits of $G$ is
$$\mathrm{ec}(G)=t_v(G)\prod_{v\in V(G)}(d^+(v)-1)!.$$
\end{thm}
There are $m^{2(m-1)}$ spanning trees of $K_{m,m}$ \cite{moh1990number}, so we conclude that there are $m^{2(m-1)}(m-1)!^{2m}$ transition vectors.
We say that a Hamilton cycle
$$H=v_{0j_1k_1}\cdots w_{0j_2k_1}\cdots v_{0j_2k_2}\cdots w_{0j_3k_2}\cdots\cdots v_{0j_1k_1}$$
(making no assumptions on the hidden portions of the vertex sequence) \textit{follows} the transition vector $t$ if $v_{0j_sk_s}\in X_{f(s)}$ and $w_{0j_{s+1}k_s}\in Y_{g(s)}$ for every $s\in[m^2]$, i.e. $k_s=f(s)$ and $j_{s+1}=g(s)$ (see \autoref{Hamilton cycle figure}).

\begin{claim} \label{Hamilton cycles following t}
For any transition vector $t$ there are exactly $(m!)^{2m(r-2)}$ Hamilton cycles in $G$ which follow $t$.
\end{claim}
\begin{proof}
    We consider any component $X_i$ of $G[V_0,\ldots,W_0]$. Each time a Hamilton path $H$ which follows $t$ visits a vertex $v\in V_0\cap X_i$, we must choose a path in $X_i$ from $v$ to the unique vertex $w\in W_0\cap Y_j$ where $Y_j$ is the component indicated by $t$ via the subword $X_iY_j$. (We know $w$ is unvisited since if it was visited previously then $X_iY_j$ must have already occurred in $t$, as $X_i\cap Y_j\cap W_0=\{w\}$.) The first time that $V_0\cap X_i$ is visited there are $m^{r-2}$ choices for such a path (only the last edge is forced), the second time there are $(m-1)^{r-2}$ choices, and so on, so that varying the paths taken inside $X_i$ gives $m^{r-2}\cdots 1^{r-2}=(m!)^{r-2}$ total choices. Similarly there are $(m!)^{r-2}$ total choices for the paths inside each $Y_j$, so considering all components together we arrive at $((m!)^{r-2})^{2m}$ Hamilton paths.
\end{proof}

\begin{figure}[h] 
\begin{center}

\begin{tikzpicture}[scale=0.5]

\filldraw [black] (-4,4.5) circle (0.2);
\filldraw [black] (-4,3.5) circle (0.2);
\filldraw [black] (-4,0.5) circle (0.2);
\filldraw [black] (-4,-0.5) circle (0.2);
\filldraw [black] (0,4.5) circle (0.2);
\draw [above] (0,4.6) node{$X_1$};
\filldraw [black] (0,3.5) circle (0.2);
\draw [below] (0,3.4) node{$X_2$};
\filldraw [black] (0,0.5) circle (0.2);
\draw [above] (0,0.6) node{$X_1$};
\filldraw [black] (0,-0.5) circle (0.2);
\draw [below] (0,-0.6) node{$X_2$};
\filldraw [black] (4,4.5) circle (0.2);
\filldraw [black] (4,3.5) circle (0.2);
\filldraw [black] (4,0.5) circle (0.2);
\filldraw [black] (4,-0.5) circle (0.2);
\filldraw [black] (8,4.5) circle (0.2);
\filldraw [black] (8,3.5) circle (0.2);
\filldraw [black] (8,0.5) circle (0.2);
\filldraw [black] (8,-0.5) circle (0.2);
\filldraw [black] (12,4.5) circle (0.2);
\draw [above] (12,4.6) node{$Y_1$};
\filldraw [black] (12,3.5) circle (0.2);
\draw [below] (12,3.4) node{$Y_1$};
\filldraw [black] (12,0.5) circle (0.2);
\draw [above] (12,0.6) node{$Y_2$};
\filldraw [black] (12,-0.5) circle (0.2);
\filldraw [black] (16,4.5) circle (0.2);
\draw [below] (12,-0.6) node{$Y_2$};
\filldraw [black] (16,3.5) circle (0.2);
\filldraw [black] (16,0.5) circle (0.2);
\filldraw [black] (16,-0.5) circle (0.2);

\draw [very thick,->] (-4,4.5) -- (-0.2,4.5);
\draw [dashed] (-4,4.5) -- (-0.2,3.5);
\draw [dashed] (-4,3.5) -- (-0.2,4.5);
\draw [very thick,->] (-4,3.5) -- (-0.2,3.5);
\draw [dashed] (-4,0.5) -- (-0.2,0.5);
\draw [very thick,->] (-4,0.5) -- (-0.2,-0.5);
\draw [very thick,->] (-4,-0.5) -- (-0.2,0.5);
\draw [dashed] (-4,-0.5) -- (-0.2,-0.5);

\draw [very thick,->] (0,4.5) -- (4-0.2,4.5);
\draw [dashed] (0,4.5) -- (4-0.1,0.6);
\draw [very thick,->] (0,3.5) -- (4-0.2,3.5);
\draw [dashed] (0,3.5) -- (4-0.1,-0.4);
\draw [dashed] (0,0.5) -- (4-0.1,4.4);
\draw [very thick,->] (0,0.5) -- (4-0.2,0.5);
\draw [dashed] (0,-0.5) -- (4-0.1,3.4);
\draw [very thick,->] (0,-0.5) -- (4-0.2,-0.5);

\draw [very thick,->] (4,4.5) -- (8-0.2,4.5);
\draw [dashed] (4,4.5) -- (8-0.1,0.6);
\draw [dashed] (4,3.5) -- (8-0.2,3.5);
\draw [very thick,->] (4,3.5) -- (8-0.1,-0.4);
\draw [dashed] (4,0.5) -- (8-0.1,4.4);
\draw [very thick,->] (4,0.5) -- (8-0.2,0.5);
\draw [very thick,->] (4,-0.5) -- (8-0.1,3.4);
\draw [dashed] (4,-0.5) -- (8-0.2,-0.5);

\draw [very thick,->] (8,4.5) -- (12-0.2,4.5);
\draw [dashed] (8,4.5) -- (12-0.1,0.6);
\draw [very thick,->] (8,3.5) -- (12-0.2,3.5);
\draw [dashed] (8,3.5) -- (12-0.1,-0.4);
\draw [dashed] (8,0.5) -- (12-0.1,4.4);
\draw [very thick,->] (8,0.5) -- (12-0.2,0.5);
\draw [dashed] (8,-0.5) -- (12-0.1,3.4);
\draw [very thick,->] (8,-0.5) -- (12-0.2,-0.5);

\draw [dashed] (12,4.5) -- (16-0.2,4.5);
\draw [very thick,->] (12,4.5) -- (16-0.2,3.5);
\draw [very thick,->] (12,3.5) -- (16-0.2,4.5);
\draw [dashed] (12,3.5) -- (16-0.2,3.5);
\draw [very thick,->] (12,0.5) -- (16-0.2,0.5);
\draw [dashed] (12,0.5) -- (16-0.2,-0.5);
\draw [dashed] (12,-0.5) -- (16-0.2,0.5);
\draw [very thick,->] (12,-0.5) -- (16-0.2,-0.5);

\draw [very thick] (16,4.5) -- (18,4.5);
\draw [dashed] (16,4.5) -- (18,4);
\draw [dashed] (16,3.5) -- (18,4);
\draw [very thick] (16,3.5) -- (18,3.5);
\draw [very thick] (16,0.5) -- (18,0.5);
\draw [dashed] (16,0.5) -- (18,0);
\draw [dashed] (16,-0.5) -- (18,0);
\draw [very thick] (16,-0.5) -- (18,-0.5);

\draw [very thick,->] (-6,4.5) -- (-4.2,4.5);
\draw [dashed,->] (-6,4) -- (-4.1,3.5);
\draw [dashed,->] (-6,4) -- (-4.2,4.5);
\draw [very thick,->] (-6,3.5) -- (-4.1,3.5);
\draw [very thick,->] (-6,0.5) -- (-4.2,0.5);
\draw [dashed,->] (-6,0) -- (-4.1,-0.5);
\draw [dashed,->] (-6,0) -- (-4.2,0.5);
\draw [very thick,->] (-6,-0.5) -- (-4.1,-0.5);

\draw (-4,-1.9) node{$W_2$};
\draw (-0,-1.9) node{$V_0$};
\draw (4,-1.9) node{$V_1$};
\draw (8,-1.9) node{$V_2$};
\draw (12,-1.9) node{$W_0$};
\draw (16,-1.9) node{$W_1$};

\end{tikzpicture}
\caption{The solid lines describe a Hamilton cycle in $G_{4,2}$ which follows the transition vector $X_1Y_1X_2Y_2X_1Y_2X_2Y_1$.} \label{Hamilton cycle figure}
\end{center}
\end{figure}

Note that every Hamilton cycle follows exactly one transition vector. Therefore, the number of Hamilton cycles in $G$ is
$$\begin{aligned}m^{2(m-1)}(m-1)!^{2m}m!^{2m(r-2)}=m^{(2r-2)m^2+O(m^2/\log m)}=n^{n/2+O(n/\log n)}.\end{aligned}$$

It follows there is a family $\mathcal P$ of $n^{n/2+O(n/\log n)}$ Hamilton paths in $G$. Suppose $P,Q\in\mathcal P$ and $P\cup Q$ contains a 2-part $\ell$-cycle $C$. Since $G_{r,m}$ contains no $C_{r,r}$, and $C_{r,r}$ is the unique 2-part cycle of type 0, we have $t(C)\ne 0$. We will filter out these remaining $\ell$-cycles. Partition $[n]$ into equal parts $N=N_0\sqcup\cdots\sqcup N_{2r}$. Let $\Sigma$ be the set of all Hamilton paths starting in $N_0$ and whose $i^{th}$ vertex belongs to $N_{i\pmod{2r+1}}.$ Clearly no two paths in $\Sigma$ create an unbalanced $\ell$-cycle, and
    $$|\Sigma|=\left(\frac{n}{2r+1}\right)^{2r+1}\left(\frac{n}{2r+1}-1\right)^{2r+1}\cdots(1)^{2r+1}=(n/(2r+1))!^{2r+1}=n^{n+O(n/\log n)}.$$
Let $\pi$ be a random relabeling of $[n]$, then taking an outcome in which $|\pi\mathcal P\cap\Sigma|$ is at least average, we obtain a family $\mathcal P'$ of Hamilton paths no two of which create any two-part cycle, with $|\mathcal P'|=n^{n/2+O(n/\log n)}$. To convert this to an upper bound on $\hat M(n,\ell)$ we refer to a folklore lemma about vertex-transitive graphs (see e.g. \cite{godsil2001algebraic}, Lemma 7.2.2) 
\begin{lemma} \label{Vertex transitive lemma}
If a graph $G$ is vertex-transitive, then
$$\alpha(G)\omega(G)\le|V(G)|.$$
\end{lemma}
 Consider the graph whose vertices are Hamilton paths on $[n]$ where two paths are adjacent if they create a two-part $\ell$-cycle. Then $\mathcal{P}'$ corresponds to an independent set. Applying \autoref{Vertex transitive lemma} to this graph, we obtain $$\hat M(n,\ell)\le\frac{n!}{n^{n/2+O(n/\log n)}}=(n!)^{1/2+O(1/\log n)}.$$
Now consider general $n\in\mathbb N$. Note that $\hat M(n,\ell)$ is increasing. Thus taking the smallest $n'>n$ satisfying $(2r-2)(2r+1)|n'$ gives
$$\hat M(n,\ell)\le (n+O(1))!^{1/2+O(1/\log n)}=(n!)^{1/2+O(1/\log n)}.$$
\qed

\section{Concluding remarks} \label{Concluding remarks}

As noted in \autoref{Introduction}, taking all matchings in a $C_4$-free bipartite graph does not give rise to an $S_2'$-set in the symmetric group. Instead, it obtains a family $\mathcal F$ of permutations satisfying the weaker condition that for all $\alpha,\beta,\gamma,\delta\in\mathcal F$, 
    \begin{equation} \label{sidon implication}
    \alpha\beta^{-1}=\gamma\delta^{-1}\Longrightarrow\forall i\ [\alpha(i)=\beta(i)\text{ and }\gamma(i)=\delta(i)]\text{ or }[\alpha(i)=\gamma(i)\text{ and }\beta(i)=\delta(i)].
    \end{equation}
We attempted to improve \autoref{Probabilistic constructions} in the case $\Gamma=S_n$ by intersecting $\mathcal F$ with a family $\mathcal G\subseteq S_n$ such that for all distinct $\alpha,\beta,\gamma,\delta\in\mathcal G$ there exists $i\in[n]$ such that $|\{\alpha(i),\beta(i),\gamma(i),\delta(i)\}|\ge 3$. However, Bukh and Keevash \cite{Bukh2025generalization} proved the following theorem that generalizes the upper bound of Blackburn and Wild \cite{blackburn1998optimal} on perfect hash codes.


\begin{thm}[Bukh and Keevash \cite{Bukh2025generalization}]\label{BK theorem}
Suppose that $S\subseteq [q]^n$ is family of words such that among every $t$ words there is a coordinate
with at least $v$ values. Then $|{S}|\leq \binom{t}{2}q^{(1-\frac{v-2}{t-1})n}$.
\end{thm}
Before proving the theorem, a lemma is needed.
\begin{lemma}\label{fam}
  There is a family $\mathcal{F}\subset \binom{[t-1]}{v-2}$ of size $|{\mathcal{F}}|=t-1$ such that every element of $[t-1]$
  is in exactly $v-2$ sets of $\mathcal{F}$.
\end{lemma}
\begin{proof}
Let $\mathcal{F}$ consist of cyclic shifts of $[v-2]$ modulo $t-1$.
\end{proof}
\begin{proof}[Proof of Theorem \ref{BK theorem}]
Let $\mathcal{F}=\{I_1,\dotsc,I_{t-1}\}$ be the family as in Lemma \ref{fam}. Cut each word $w\in S$ into $t-1$ consecutive subwords $w_1,...,w_{t-1}$ of length
$n/(t-1)$ each. For a set $I\in \mathcal{F}$, define $w_I$ to be the
concatenation of the words $(w_i)_{i \in [t-1]\setminus I}$. So, $w_I$ is a word of length
$(1-\frac{v-2}{t-1})n$.

Do the following for as long as possible: if there is a pair $(j,u)\in [t-1]\times [q]^{(1-\frac{v-2}{t-1})n}$ such that
the set $S_{j,u}:= \{w\in S: w_{I_j}=u\}$ has at most $j$ elements, remove all elements of $S_{j,u}$ from $S$. Note
that each pair $(j,u)$ occurs at most once in this process. So, the total number of words removed from $S$ is
at most $\binom{t}{2}q^{(1-\frac{v-2}{t-1})n}$.

We claim that $S$ is now empty. Indeed, suppose that some word $w$ survived to the end of this process.
For each $j=1,2,\dotsc,t-1$ in order, find a word $w^{(j)}\in S$ such that $w^{(j)}_{I_j}=w_{I_j}$ and such that
$w^{(j)}$ is distinct from previously selected words $w,w^{(1)},\dotsc,w^{(j-1)}$. The latter is possible because survival of $w$ implies
$|{S_{j,w_{I_j}}}|>j$.

The definition of $\mathcal{F}$ implies that in each coordinate the $t$ words $w,w^{(1)},\dotsc,w^{(t-1)}$ take at most $v-1$ values.
As the words are distinct, we reached a contradiction.
\end{proof}

This implies that our approach only proves that $M_2'(S_n)\ge(n!)^{1/6+O(1/\log n)}$. We believe that such `$t$-wise $v$-different codes' may be of some independent interest.

We considered whether the idea in our construction of $S_2$-sets in $S_n\times S_n$ could be generalized to give $S_2'$-sets or to give $S_k$-sets for $k\ge 3$. For $S_2'$-sets, we looked for constructions taken from the set $B=\{(f(\alpha),g(\alpha)):\alpha\in S_n\}$, where $f(\alpha)$ and $g(\alpha)$ are some words on $\alpha$ and some fixed permutations. It seems to us that for any choice of $f,g$, the equations of the form $x_1y_1^{-1}x_2y_2^{-1}=1$ with variables in $B$ either simplify to a single Sidon equation in $S_n$, or are too complicated to usefully employ the choice of $f,g$. We were also unable to find any similar construction that works for $S_k$-sets ($k\ge 3$). It may be interesting to see whether there is a natural construction of $S_k$-sets in $S_n^k$, extending our loose analogy with the abelian constructions.

Besides the constructions used in \autoref{Sn times Sn result} we found other $S_2$-sets of the same size. Let $\pi',\sigma'$ be two permutations of $[n]$ such that $\pi:=(\pi')^2,\sigma:=(\sigma')^2$ are both derangements and involutions, and such that $\sigma\pi=\rho_1\rho_2$ for two disjoint $(n/2)$-cycles $\rho_1,\rho_2$. Then one can show that $\{(\pi'\alpha\pi',\sigma'\alpha\sigma'):\alpha\in S_n,\alpha(1)=1\}$ is an $S_2$-set in $S_n\times S_n$, and in fact it is also a special case of \autoref{recipe}.

It is interesting that the proof of \autoref{Sk upper bound} does not work when $k$ is odd. In fact, if $k=2r+1$ then as in the proof of \autoref{Sk upper bound} one can define $L=\{\alpha_1\cdots\alpha_{r+1}:\alpha_i\in A\}$ and show that $L$ is a near-optimal $S_2[|A|]$-set. However, when $g\ge 2$ it is possible for large $S_2[g]$-sets to exist in abelian groups, so only the final step in the proof fails.

We list some open questions:
\begin{itemize}
    \item[(1)] 
    For each $k\ge2$ do there exist constants $C,c$ such that
    $$M_k(\Gamma)\le C\text{ or }M_k(\Gamma)\ge c|\Gamma|^{1/k}$$
    holds for every finite group $\Gamma$?
    \item[(2)] Does \autoref{Sk upper bound} extend to the case that $k$ is odd?
    \item[(3)] Improve the lower or upper bounds in the inequalities
    $$(n!)^{1/k-O(1/\log n)}\le M_k(S_n)<(n!)^{1/k}$$
    and
    $$(n-1)!\le M_2(S_n\times S_n)<n!.$$
\end{itemize}

\section{Acknowledgements}

The authors would like to thank Boris Bukh and Peter Keevash for the proof of Theorem \ref{BK theorem}. This research was partially supported by NSF DMS-2245556.


\bibliographystyle{plain}
	\bibliography{references.bib}

\end{document}